\documentclass[11pt]{article}
\usepackage{amsfonts}
\usepackage{latexsym,amssymb,amsfonts,graphicx}
\usepackage{amsmath,dsfont}
\usepackage[linewidth=1pt]{mdframed}
\usepackage{mathrsfs}
\usepackage[normalem]{ulem}
\usepackage{epsfig}
\usepackage{tabularx}
\usepackage{geometry}
\usepackage{verbatim}
\usepackage{enumitem}

\usepackage{url}
\usepackage{bm}
\usepackage{color}
\usepackage{natbib}

\usepackage[unicode=true,colorlinks,citecolor=linkcolor,linkcolor=blue,urlcolor=blue]{hyperref}

\providecommand{\U}[1]{\protect\rule{.1in}{.1in}}

\newtheorem{thm}{Theorem}[section]

\newtheorem{lem}{Lemma}[section]

\newtheorem{prop}{Proposition}[section]

\usepackage[bf,SL,BF]{subfigure}

\newenvironment{proof}[1][Proof]{\noindent\textbf{#1.} }{\ \rule{0.5em}{0.5em}}
\normalsize
\numberwithin{equation}{section}
\allowdisplaybreaks[4]

\definecolor{linkcolor}{rgb}{0,0,0.7}
\begin{document}

\title{Optimal dividend and capital injection under spectrally positive Markov additive models}
\author{Wenyuan Wang\thanks{School of Mathematical Sciences, Xiamen University, Xiamen, Fujian, China. Email: wwywang@xmu.edu.cn}\,\,\,\,\,\,\,\,
Kaixin Yan\thanks{School of Mathematical Sciences, Xiamen University, Xiamen, Fujian, China. Email: kaixinyan@stu.xmu.edu.cn}
\,\,\,\,\,\,\,\,Xiang Yu\thanks{Department of Applied Mathematics, The Hong Kong Polytechnic University, Hung Hom,
Kowloon, Hong Kong. E-mail: xiang.yu@polyu.edu.hk}
}
\date{
%May 2022
}

\date{\vspace{-0.7in}}

\maketitle

\begin{abstract}
This paper studies De Finetti's optimal dividend problem with capital injection under spectrally positive Markov additive models. Based on dynamic programming principle, we first study an auxiliary singular control problem with a final payoff at an exponential random time. The double barrier strategy is shown to be optimal and the optimal barriers are characterized in analytical form using fluctuation identities of spectrally positive L\'{e}vy processes. We then transform the original problem under spectrally positive Markov additive models into an equivalent series of local optimization problems with the final payoff at the regime-switching time. The optimality of the regime-modulated double barrier strategy can be confirmed for the original problem using results from the auxiliary problem and the fixed point argument for recursive iterations.

\ \\
\vspace{-0.2in}
\end{abstract}

\vspace{-0.2in}

\section{Introduction}
\vspace{-0.2in}
De Finetti's optimal dividend problem with capital injection has become a fast-growing research topic in insurance and corporate finance. The goal of the optimal control problem is to maximize the expected net present value (NPV) of dividends when the shareholders inject capital whenever necessary over an infinite horizon to bail out the company from ruin. In particular, abundant studies can be found when underlying risk processes follow general spectrally positive or spectrally negative L\'evy processes; see among \cite{Avram07}, \cite{Avanzi11}, \cite{Bay13}, \cite{ZWYC15}, \cite{ZCY17}, \cite{ZWY}, \cite{YP2017}, \cite{YP17b}, \cite{PYY2018}, \cite{NPYY}, \cite{WW22} and references therein.

On the other hand, the regime-switching model has been widely used thanks to its capability to capture the changes or transitions of market trends. A large amount of empirical studies on regime switching can be found in the literature; see, for example, \cite{Ham}, \cite{So}, \cite{AngBeK02b}, \cite{Pelle}, \cite{AngTim}. In the context of optimal dividend with regime switching, we refer to a short list of \cite{JiaPis2012}, \cite{Azcue15}, \cite{WeiWY} and \cite{ZY2016}. Recently, the optimality of a regime-modulated refraction-reflection strategy is verified in the optimal dividend problem with capital injection under spectrally negative Markov additive processes in  \cite{NPY2020} when the dividend process is assumed to be absolutely continuous with a bounded rate. The spectrally negative Markov additive process can be understood as a family of spectrally negative L\'{e}vy processes switching via an independent Markov chain, in which a negative jump occurs when there is a regime change. The jump size random variable is independent of the family of L\'evy processes and the Markov chain, which can be understood as the cost for the insurance company to get adapted to the new regime. To encode the Chapter 11 bankruptcy, the optimality of a barrier strategy in the optimal singular dividend control problem is addressed in \cite{WYZ21} under spectrally negative L\'evy processes with endogenous regime switching and Parisian ruin with exponential delay. However, it remains an open problem whether the optimal dividend control fits the barrier type when the risk process follows the spectrally positive Markov additive process.

The present paper fills this gap, and we verify that the regime-modulated double barrier strategy attains the optimality among all singular dividend and capital injection controls under the spectrally positive Markov additive process. To be more precise, within each regime state $i$, we can find a positive barrier $b_i^*>0$ such that: (i) when the risk process exceeds the barrier $b_i^*$, the company pays a dividend so that the surplus process reflects at the level $b_i^*$; (ii) when the risk process falls below $0$, the capital is injected to bail out the risk process from ruin. Our methodology is based on fluctuation identities of spectrally positive  L\'evy processes and fixed point arguments for recursive iterations induced by dynamic programming. Comparing with \cite{NPY2020}, we stress that our problem differs substantially as the dividend control process is not absolutely continuous and we work with spectrally positive L\'evy processes. Distinct computations and proofs are required to handle the auxiliary optimal singular control problem with a final payoff. In addition, contrary to \cite{NPY2020}, our value functions under the spectrally positive Markov additive process are unbounded due to upward jumps, causing some new difficulties in the fixed point augments. We also note that, in the auxiliary optimal dividend problem with a final payoff, the value function and the optimal barrier can be expressed in a more concise way comparing with the results in \cite{NPY2020} thanks to different properties of spectrally positive L\'evy processes.

The rest of the paper is organized as follows. In Section \ref{sec:model}, we formulate the optimal dividend and capital injection problem under the spectrally positive Markov additive process and introduce some preliminaries of spectrally positive L\'{e}vy processes. In Section \ref{sec:aux}, we study an auxiliary optimal dividend and capital injection problem with a final payoff at an exponential terminal time. The optimality of a double-barrier strategy is verified using fluctuation identities of spectrally positive L\'{e}vy processes and smooth-fit principle. In Section \ref{sec:opt}, based on dynamic programming arguments, we prove the optimality of the regime-modulated double barrier strategy in the original control problem using fixed point arguments for recursive iterations and some results from the auxiliary control problem.

\section{Problem Formulation under the Spectrally Positive Markov Additive Process}\label{sec:model}
\subsection{Problem Formulation}
\vspace{-0.2in}
Let us consider the risk process modelled by the \textit{so-called} spectrally positive Markov additive process $\{(X_{t},Y_t);t\geq 0\}$; see, for example, the detailed introduction in Section XI of \cite{Asmussen03}. Here, $\{Y_t;t\geq 0\}$ is a continuous time Markov chain with finite state space $\mathcal{E}$ and the generator matrix $(q_{ij})_{i,j\in\mathcal{E}}$. Condition on that Markov chain $Y$ is in the state $i$, the process $X$ evolves as a spectrally
positive L\'evy process $X^{i}$ until the Markov chain $Y$ switches to another state $j\neq i$, at which instant there is a downward jump in $X$ with a random amount $J_{ij}$. We assume that $(X^{i})_{i\in\mathcal{E}}$, $Y$, and $(J_{ij})_{i,j\in\mathcal{E}}$ are mutually independent and are defined on a filtered probability space $(\Omega,\mathcal{F},\{\mathcal{F}_{t};t\geq 0\},\mathrm{P})$ satisfying the usual condition. Let us denote $\mathrm{P}_{x,i}$ the law of the process $\{(X_{t},Y_t);t\geq 0\}$ conditioning on $\{X_{0}=x,Y_{0}=i\}$.

We consider a bail-out dividend control problem in this Markov additive framework, where the beneficiaries of dividends are supposed to injects capitals into the surplus process so that the resulting surplus process are always non-negative, i.e., bankruptcy never occurs. To this end, let $X$ denote the underlying risk process before dividends are deducted and capitals are injected into. We consider two non-decreasing, right-continuous, adapted processes $\{D_{t};t\geq 0\}$ and $\{R_{t};t\geq 0\}$ defined on $(\Omega,\mathcal{F},\{\mathcal{F}_{t};t\geq 0\},\mathrm{P})$, which, respectively, represent the cumulative amount of dividends and injected capitals with $D_{0-}=R_{0-}=0$. The surplus process after taking into account the dividends and capital injection is defined by $U_{t}:=X_{t}-D_{t}+R_{t}, t\geq 0$. The value function of the de Finetti's dividend control problem with capital injection is defined by
\begin{eqnarray}
\label{goal}
V(x,i)\hspace{-0.2cm}&=&\hspace{-0.2cm}
\sup_{D,R}\mathrm{E}_{x,i}\left[\int_{0}^{\infty}e^{-\int_{0}^{t}\delta_{Y_s}\mathrm{d}s}\mathrm{d}D_{t}-\phi\int_{0}^{\infty}e^{-\int_{0}^{t}\delta_{Y_s}\mathrm{d}s}\mathrm{d}R_{t}\right]\notag\\
\hspace{-0.2cm}&&\hspace{-0.2cm}
\text{subject to\ \  $U_{t}=X_{t}-D_{t}+R_{t}\geq 0$ for all $t\geq 0$,}
\nonumber\\
\hspace{-0.2cm}&&\hspace{-0.2cm}
\text{both $D_{t}$ and $R_{t}$ are non-decreasing, c\`{a}dl\`{a}g and adapted processes,}
\nonumber\\
\hspace{-0.2cm}&&\hspace{-0.2cm}
\text{$D_{0-}=R_{0-}=0$, and $\int_{0}^{\infty}e^{-\int_{0}^{t}\delta_{Y_s}\mathrm{d}s}\mathrm{d}R_{t}<\infty$, $\mathrm{P}_{x,i}$-almost surely},
\end{eqnarray}
where $(\delta_{i}) \in (0,\infty)^{\mathcal{E}}$ is a discounting rate function that switches according to the economic environment depicted by the Markov chain $Y$, and $\phi>1$ depicts the cost per unit capital injected. Our goal is to find the optimal strategy $(D^*,R^*)$ that attains the value functions $(V(x,i))_{i\in\mathcal{E}}$.

Following the similar proof of Proposition $3.4$ in \cite{NPY2020}, we can readily obtain the following dynamic programming result for the value function starting from the regime state $i$, and its proof is hence omitted.

\begin{prop}
For $x\in\mathbb{R}$ and $i\in\mathcal{E}$, we have that
\begin{align}\label{dpp}
V(x,i)=\sup_{D,R}\mathrm{E}_{x,i}\left [ \int_{0}^{e_{\lambda_{i}}}e^{-\int_{0}^{t}\delta_{Y_s}\mathrm{d}s}\mathrm{d}D_{t}-\phi\int_{0}^{e_{\lambda_{i}}}e^{-\int_{0}^{t}\delta_{Y_s}\mathrm{d}s}\mathrm{d}R_{t}+e^{-\int_0^{e_{\lambda_{i}}}\delta_{Y_s}ds}V(U_{e_{\lambda_{i}}}, Y_{e_{\lambda_{i}}})\right],
\end{align}
where $e_{\lambda_{i}}$ is the first time $Y$ switches the regime state under $\mathrm{P}_{x,i}$.

\end{prop}

The next theorem is the main result of this paper, which confirms the optimality of a regime-modulated double barrier strategy for the stochastic control problem \eqref{goal}, whose proof is deferred to Section \ref{sec:opt}.

\begin{thm}
\label{thm4.1}
%Suppose that $\max_{i\in\mathcal{E}}\int_{1}^{\infty}y\upsilon_{i}(\mathrm{d}y)<\infty$ where $\upsilon_{i}$ represents the L\'evy measure of $X^{i}$.
There exists a function $\mathbf{b}^{*}=(b_{i}^{*})_{i\in\mathcal{E}}\in(0,\infty)^{\mathcal{E}}$ such that the double barrier dividend and capital injection strategy with the dynamic upper reflection barrier $b^{*}_{Y_{t}}$ and fixed lower reflection barrier $0$ is optimal that attains the value function in \eqref{goal} that
$$V_{0,\mathbf{b}^{*}}(x,i)=V(x,i),\quad (x,i)\in\mathbb{R}_{+}\times \mathcal{E},$$
where $V_{0,\mathbf{b}^{*}}(x,i)$ represents the value function of the double barrier dividend and capital injection strategy with upper barrier $b^{*}_{Y_{t}}$ and lower barrier $0$.
\end{thm}

\subsection{Some preliminaries of spectrally positive L\'{e}vy processes}

Let $X=(X_t)_{t\geq 0}$ be a L\'{e}vy process defined on a probability space $(\Omega, \mathcal{F},\mathrm{P})$. For $x\in\mathbb{R}$, we denote by $\mathrm{P}_x$ the law of $X$ starting from $x$ and write $\mathrm{E}_x$ the associated expectation. We also use $\mathrm{P}$ and $\mathrm{E}$ in place of $\mathrm{P}_0$ and $\mathrm{E}_0$. The L\'{e}vy process $X$ is said to be \textit{spectrally positive} if it has no negative jumps and it is not a subordinator. The Laplace exponent $\psi: [0,\infty)\rightarrow \mathbb{R}$ satisfying
\begin{align*}
\mathrm{E}[e^{-\theta X_t}]=:e^{\psi(\theta)t}, \quad t,\theta\geq 0,
\end{align*}
is given by the L\'{e}vy-Khintchine formula that
\begin{align*}
\psi(\theta):=\gamma \theta+\frac{\sigma^2}{2}\theta^2+\int_{(0,\infty)}(e^{-\theta z}-1+\theta z\mathbf{1}_{\{z<1\}})\upsilon(dz), \quad \theta\geq 0,
\end{align*}
where $\gamma\in\mathbb{R}$, $\sigma\geq 0$, and $\upsilon$ is the L\'{e}vy measure of $X$ on $(0,\infty)$ that satisfies
\begin{align*}
\int_{(0,\infty)}(1\wedge z^2)\upsilon(dz)<\infty.
\end{align*}
It is well-known that $X$ has paths of bounded variation if and only if $\sigma=0$ and $\int_{(0,1)}z\upsilon(dz)<\infty$; in this case, we have
\begin{align*}
X_t=-ct+S_t,\quad t\geq 0,
\end{align*}
where
\begin{align*}
c:=\gamma+\int_{(0,1)}z\upsilon(dz),
\end{align*}
and $(S_t)_{t\geq 0}$ is a driftless subordinator. As we have ruled out the case that $X$ has monotone paths, it holds that $c>0$. Its Laplace exponent is given by
\begin{align*}
\psi(\theta)=c\theta+\int_{(0,\infty)}(e^{-\theta z}-1)\upsilon(dz),\quad \theta \geq 0.
\end{align*}
To exclude the trivial case, it is assumed throughout the paper that
\begin{align*}
\mathrm{E}[X_1]=-\psi'(0+)<\infty.
\end{align*}

Let us also recall the $q$-scale function for the spectrally positive L\'{e}vy process $X$. For $q>0$, 
%there exists a continuous and increasing function 
the $q$-scale function $W_{q}:\mathbb{R}\rightarrow [0,\infty)$
%, called 
is continuous and strictly increasing on $(0,\infty)$ and takes value zero on $(-\infty,0)$ with its Laplace transform on $[0,\infty)$ given by
\begin{align*}
\int_0^{\infty} e^{-sx}W_{q}(x)dx=\frac{1}{\psi(s)-q},\quad s>\Phi(q),
\end{align*}
where $\Phi(q):=\sup\{s\geq 0: \psi(s)=q\}$. We also define $Z_{q}(x)$ by
\begin{align*}
Z_{q}(x):=1+q\int_0^x W_{q}(y)dy,\quad x\in\mathbb{R},
\end{align*}
and its anti-derivative
\begin{align*}
\overline{Z}_{q}(x):=\int_0^x Z_{q}(y)dy,\quad x\in\mathbb{R}.
\end{align*}
We recall that if $X$ has paths of bounded variation, $W_q(x)\in C^1((0,\infty))$ if and only if the L\'{e}vy measure $\upsilon$ has no atoms. If $X$ has paths of unbounded variation, we have that $W_q(x)\in C^1((0,\infty))$. Moreover, if $\sigma>0$, we have $W_q(x)\in C^2((0,\infty))$. Hence, we have that $Z_q(x)\in C^1((0,\infty))$, $\overline{Z}_q(x)\in C^1(\mathbb{R})$ and $\overline{Z}_q(x) \in C^2((0,\infty))$ for bounded variation case; and we have $Z_q(x)\in C^1(\mathbb{R})$, $\overline{Z}_q(x) \in C^2(\mathbb{R})$ and $\overline{Z}_q(x)\in C^3((0,\infty))$ for the unbounded variation case. We also know that
\begin{align*}
\begin{split}
W_{q} (0+) &= \left\{ \begin{array}{ll} 0 & \textrm{if $X$ is of unbounded
variation,} \\ 1/c & \textrm{if $X$ is of bounded variation.}
\end{array} \right.
\end{split}
\end{align*}

Let us define $\tau_{a}^{-}:=\inf\{t\geq 0; X_{t}<a\}$ and $\tau_{b}^{+}:=\inf\{t\geq 0; X_{t}>b\}$. Then, for $b\in(0,\infty)$ and $x\in[0,b]$, we have
\begin{eqnarray}
\label{fluc.1}
\mathrm{E}_{x}\left[e^{-q \tau_{0}^{-}}\mathbf{1}_{\{\tau_{0}^{-}<\tau_{b}^{+}\}}\right]
\hspace{-0.3cm}&=&\hspace{-0.3cm}
\frac{W_{q}(b-x)}{W_{q}(b)},
\\
\label{fluc.2}
\mathrm{E}_{x}\left[e^{-q \tau_{b}^{+}}\mathbf{1}_{\{\tau_{b}^{+}<\tau_{0}^{-}\}}\right]
\hspace{-0.3cm}&=&\hspace{-0.3cm}
Z_{q}(b-x)-\frac{Z_{q}(b)}{W_{q}(b)}W_{q}(b-x).
\end{eqnarray}

\section{Auxiliary Optimal Dividend Problem with A Final Payoff}\label{sec:aux}

In this section, we first consider an auxiliary optimal dividend and capital injection problem with a final payoff at an independent exponential terminal time in a single spectrally positive L\'{e}vy model. Let $X_t$ be the underlying risk process that follows a single spectrally positive L\'{e}vy process and let $(D_t, R_t)_{t\geq 0}$ denote nondecreasing, right continuous and $\mathcal{F}_t$-adapted dividend and capital injection control processes starting from zero. The controlled surplus process $U_t$ is defined by $U_t:=X_t-D_t+R_t$ and it is required that $U_t\geq 0$ a.s. for all $t\geq 0$. Let $\omega(x)$ be a final payoff function and the terminal time exponential random variable is denoted by $e_{\lambda}$ with parameter $\lambda$.

Throughout this section, we assume that the payoff function $\omega$ is continuous and concave over $[0, \infty)$ with
$\omega^{\prime}_{+}(0+)\leq \phi$ and $\omega^{\prime}_{+}(\infty)\in[0,1]$ where $\omega^{\prime}_{+}(x)$ denotes
the right derivative of $\omega$ at $x$.
For $\delta>0$, $\lambda>0$ and $q=\delta+\lambda$, the expected net present value (NPV) of the dividends and capital injections with a final payoff at the random time $e_{\lambda}$ is defined by
\begin{eqnarray}\label{lap.pari.T}
V_{D,R}^{\omega}(x)\hspace{-0.3cm}&:=&\hspace{-0.3cm}
\mathrm{E}_{x}\bigg[\int_0^{e_\lambda}e^{-\delta t}\mathrm{d}D_t-
\phi \int_{0}^{e_\lambda}e^{-\delta t}\mathrm{d} R_t+e^{-\delta e_\lambda}\omega(U_{e_\lambda})\bigg]
\nonumber\\
\hspace{-0.3cm}&=&\hspace{-0.3cm}
\mathrm{E}_{x}\bigg[\int_{0}^{\infty}\lambda e^{-\lambda s}
\bigg[\int_0^{s}e^{-\delta t}\mathrm{d}D_t-
\phi\int_0^{s}e^{-\delta t}\mathrm{d}R_t+e^{-\delta s}\omega(U_{s})\bigg]
\mathrm{d}s\bigg]
\nonumber\\
\hspace{-0.3cm}&=&\hspace{-0.3cm}
\mathrm{E}_{x}\bigg[\int_0^{\infty}e^{-q t}\mathrm{d}D_t-\phi\int_0^{\infty}e^{-q t}\mathrm{d}R_t+\lambda\int_{0}^{\infty} e^{-q t}\omega(U_{t})
\mathrm{d}t\bigg].
\end{eqnarray}
The value function of the auxiliary stochastic control problem is then given by
\begin{eqnarray}\label{valfun}
V_{D^*,R^*}^{\omega}(x)=\sup_{D,R} V_{D,R}^{\omega}(x)
\hspace{-0.3cm}&&\hspace{-0.3cm}\,\,\,
\text{subject to\ \  $U_{t}=X_{t}-D_{t}+R_{t}\geq 0$ for all $t\geq 0$,}
\nonumber\\
\hspace{-0.3cm}&&\hspace{-0.3cm}\,\,\,
\text{both $D_{t}$ and $R_{t}$ are non-decreasing, c\`{a}dl\`{a}g and adapted processes,}
\nonumber\\
\hspace{-0.3cm}&&\hspace{-0.3cm}\,\,\,
\text{$D_{0-}=R_{0-}=0$, and $\int_0^{\infty}e^{-q t}\mathrm{d}R_t<\infty$, $\mathrm{P}_{x}$-almost surely}.
\end{eqnarray}

We first follow some heuristic arguments to locate the optimal singular control in some smaller subset of admissible dividend and capital injection strategies. In fact, due to the time value of money (i.e., $q>0$), it seems reasonable to inject capitals as late as possible. In addition, due to the transaction costs charged for each unit of capitals injected ($\phi>1$), whenever capitals injection is required, the injected capital should be the amount to keep the surplus process non-negative, i.e., the surplus process will reflect from below at $0$. The above intuitive arguments motivate the following Lemma \ref{R.continu.}. The proof is essentially similar to that of Lemma 4.2 in \cite{WW22} and is hence omitted.

\begin{lem}
\label{R.continu.}
The optimal dividend and capital injection process $\{(D_{t}, R_{t});t\geq 0\}$ for the optimization problem \eqref{valfun} is such that $0\leq \Delta D_{t}\leq X_{t-}$ and
\begin{eqnarray}\label{R.form}
R_{t}=-\inf_{s\leq t}(X_{s}-D_{s})\wedge 0.
\end{eqnarray}
In particular, $\{R_{t};t\geq 0\}$ is continuous.
\end{lem}

Thanks to Lemma \ref{R.continu.}, we are able to restrict ourselves to continuous capital injection process of the form \eqref{R.form} when characterizing the optimal dividend control to the problem \eqref{valfun}.
%In the sequel, we also would like to consider narrowing down the set of dividend processes in which the dividend component of the optimal dividend and capital injection strategy lies.
To continue, motivated by many existing studies in optimal dividend problems; see, for example, \cite{Azcue15}, \cite{Avram07}, \cite{YP17b}, \cite{PYY2018}, \cite{Loe08}, \cite{Loe09}, \cite{WW22}, \cite{ZCY17}, we conjecture that the optimal singular dividend control in   \eqref{valfun} also fits the barrier type. Recall that a barrier dividend strategy with the barrier $b\in(0,\infty)$ refers to the singular dividend control process that any excessive surplus above $b$ is deducted such that the resulting surplus process reflects from above at the level $b$.

Based on our conjecture, let us first work with those dividend and capital injection strategies when dividends are paid according to an upper barrier $b$ and capitals are injected according to the form \eqref{R.form}, which we shall refer as the double barrier strategy with barriers $(0,b)$, denoted by $(D_t^{(0,b)}, R_t^{(0,b)})_{t\geq 0}$. That is, under a double barrier strategy  $(D_t^{(0,b)}, R_t^{(0,b)})_{t\geq 0}$, the controlled surplus process will be reflected from above at $b$ whenever it is to up-cross the level $b$, and will be reflected from below at $0$ whenever it is to down-cross the level $0$. Let us also denote by $U^{(0,b)}:=X_t-D_t^{(0,b)}+R_t^{(0,b)}$ the resulting surplus process. The expected NPV with a final payoff under a double barrier strategy is defined by
\begin{align}\label{expdouble}
V^{\omega}_{0,b}(x):=\mathrm{E}_{x}\bigg[\int_0^{\infty}e^{-q t}\mathrm{d}D^{(0,b)}_t-\phi\int_0^{\infty}e^{-q t}\mathrm{d}R^{(0,b)}_t+\lambda\int_{0}^{\infty} e^{-q t}\omega(U^{(0,b)}_{t})
\mathrm{d}t\bigg].
\end{align}
In our previous conjecture, the value function \eqref{valfun} can be attained by a double barrier strategy with some particularly chosen $0$ and $b^{\omega}$, i.e. $V_{D^*,R^*}^{\omega}(x)=V^{\omega}_{0,b^{\omega}}(x)$. To verify this conjecture and characterize $b^{\omega}$, let us first express the expected NPV in \eqref{expdouble} with an arbitrary upper reflecting barrier $b\in(0,\infty)$.

\begin{prop}
\label{prop2.1}
For a given $b\in(0,\infty)$, the expected NPV in \eqref{expdouble} can be written as
\begin{eqnarray}
\label{v0b(x)}
V_{0,b}^{\omega}(x)
\hspace{-0.3cm}&=&\hspace{-0.3cm}
\left\{
\small
\begin{aligned}
&
-
\overline{Z}_{q}(b-x)
-\frac{\psi^{\prime}(0+)}{q}
+
\frac{\lambda}{q}\bigg[\omega(0)+
\int_{0}^{b}\omega_{+}^{\prime}(y)
Z_{q}(y-x)\mathrm{d}y\bigg]&
\\&
+\frac{Z_{q}(b-x)}{qW_{q}(b)}
\Big[Z_{q}(b)-\phi
-\lambda\int_{0}^{b}\omega_{+}^{\prime}(y)W_{q}(y)
\mathrm{d}y\Big],& x\in[0,b],\\
&x-b+V_{0,b}^{\omega}(b),& x\in(b,\infty),\\
&\phi x+V_{0,b}^{\omega}(0), &x\in(-\infty,0).
\end{aligned}
\right.
\end{eqnarray}
\end{prop}

\begin{proof}
By Theorem 1 in \cite{Pistorius03}, one has
\begin{eqnarray}
\label{pot.dou.ref.}
\mathrm{E}_{x}\bigg[\int_{0}^{\infty} e^{-q t}\omega(U^{(0,b)}_{t})
\mathrm{d}t\bigg]
\hspace{-0.3cm}&=&\hspace{-0.3cm}
\int_{0}^{b}\omega(y)\bigg[\frac{Z_{q}(b-x)W_{q}^{\prime+}(y)}{qW_{q}(b)}-W_{q}(y-x)\bigg]\mathrm{d}y
\nonumber\\
\hspace{-0.3cm}&&\hspace{-0.3cm}
+\omega(0)\frac{Z_{q}(b-x)W_{q}(0+)}{qW_{q}(b)},\quad x\in[0,b].
\end{eqnarray}
Using (4.3) and (4.4) of \cite{Avram07}, we can get that
\begin{eqnarray}
&&\mathrm{E}_{x}\bigg[\int_0^{\infty}e^{-q t}\mathrm{d}D^{(0,b)}_t\bigg]=
-\overline{Z}_{q}(b-x)-\frac{\psi^{\prime}(0+)}{q}+\frac{Z_{q}(b)}{qW_{q}(b)}Z_{q}(b-x)
,\quad x\in[0,b],
\\
&&\mathrm{E}_{x}\bigg[\int_0^{\infty}e^{-q t}\mathrm{d}R^{(0,b)}_t\bigg]=
\frac{Z_{q}(b-x)}{qW_{q}(b)},\quad x\in[0,b].
\label{exp.cap.dou.bar.}
\end{eqnarray}
Considering \eqref{pot.dou.ref.}-\eqref{exp.cap.dou.bar.} and then rearranging terms,  we deduce that
\begin{eqnarray}
V_{0,b}^{\omega}(x)
%\hspace{-0.3cm}&=&\hspace{-0.3cm}-\overline{Z}_{q}(b-x)-\frac{\psi^{\prime}(0+)}{q}+\frac{Z_{q}(b)}{qW_{q}(b)}Z_{q}(b-x)-\frac{\phi Z_{q}(b-x)}{qW_{q}(b)}\nonumber\\\hspace{-0.3cm}&&\hspace{-0.3cm}+\lambda\left[\omega(0)\frac{Z_{q}(b-x)W_{q}(0+)}{qW_{q}(b)}+\int_{0}^{b}\omega(y)\left[\frac{Z_{q}(b-x)W_{q}^{\prime+}(y)}{qW_{q}(b)}-W_{q}(y-x)\right]\mathrm{d}y\right]\nonumber\\
\hspace{-0.3cm}&=&\hspace{-0.3cm}
-\overline{Z}_{q}(b-x)-\frac{\psi^{\prime}(0+)}{q}+\frac{Z_{q}(b)}{qW_{q}(b)}Z_{q}(b-x)-\frac{\phi Z_{q}(b-x)}{qW_{q}(b)}
\nonumber\\
\hspace{-0.3cm}&&\hspace{-0.3cm}
+\lambda\left[\frac{\omega(0)}{q}-\int_{0}^{b}\omega_{+}^{\prime}(y)\left[\frac{Z_{q}(b-x)W_{q}(y)}{qW_{q}(b)}-\frac{Z_{q}(y-x)}{q}\right]\mathrm{d}y\right],
\end{eqnarray}
which is the desired result.
\end{proof}

\begin{lem}\label{lem2.1}
Let us define $b^{\omega}\in(0,\infty)$ as the unique solution of the equation
\begin{eqnarray}
\label{smoothcond.1}
Z_{q}(x)-\phi
-\lambda\int_{0}^{x}\omega_{+}^{\prime}(y)W_{q}(y)
\mathrm{d}y=0.
\end{eqnarray}
We have $b^{\omega}>\sup\{x\geq0; q-\lambda\omega_{+}^{\prime}(x)\leq 0\}\vee Z_{q}^{-1}(\phi)>0$.
Then $V_{0,b^{\omega}}^{\omega}(x)$ is continuously differentiable on $(-\infty,\infty)$.
%if $X$ has paths of bounded variation,
Furthermore, if $X$ has paths of unbounded variation, $V_{0,b^{\omega}}^{\omega}(x)$ is twice continuously differentiable on $(0,\infty)$.
\end{lem}

\begin{proof} We first verify that \eqref{smoothcond.1} indeed admits a unique solution on $(0,\infty)$. By the concavity of $\omega$ with %$\omega_{+}^{\prime}(0^+)\leq \phi$ and
$\omega_{+}^{\prime}(\infty)\in[0,1]$, the function
$$\ell(x):=Z_{q}(x)
-
\lambda
\int_{0}^{x}\omega_{+}^{\prime}(y)
W_{q}(y)\mathrm{d}y-\phi,\quad x\in[0,\infty),$$
is first decreasing and then strictly increasing in $x$ due to the fact that it's right derivative
$$\ell_+^{\prime}(x)= W_{q}(x)\left(q-\lambda\omega_{+}^{\prime}(x)\right),\quad x\in(0,\infty),$$
is first non-positive and then positive and tends to $\infty$ as $x$ goes to $\infty$.
As a consequence, one has $\ell(\infty)=\infty$, which combined with the fact that $\ell(0)=1-\phi<0$ yields that there should be a unique zero of $\ell(x)$, i.e., \eqref{smoothcond.1} has a unique solution $b^{\omega}\in(0,\infty)$.
In addition, it is easy to see from the above arguments and \eqref{smoothcond.1} that
$$b^{\omega}>\sup\{x\geq0; q-\lambda\omega_{+}^{\prime}(x)\leq 0\}\vee Z_{q}^{-1}(\phi)>0.$$
By Proposition \ref{prop2.1}, one can derive that
\begin{eqnarray}
\label{V'(x)}
V_{0,b}^{\omega\prime}(x)
\hspace{-0.3cm}&=&\hspace{-0.3cm}
\frac{W_{q}(b-x)}{W_{q}(b)}
\bigg[\phi
+\lambda\int_{0}^{b}\omega_{+}^{\prime}(y)W_{q}(y)
\mathrm{d}y-Z_{q}(b)\bigg]
\nonumber\\
\hspace{-0.3cm}&&\hspace{-0.3cm}
+Z_{q}(b-x)
-\lambda\int_{0}^{b}\omega_{+}^{\prime}(y)W_{q}(y-x)\mathrm{d}y, \quad x\in[0,b].
\end{eqnarray}
Combing \eqref{V'(x)} and the facts that $b^{\omega}$ is the unique solution of \eqref{smoothcond.1} and $\lim\limits_{x\uparrow b}\int_{0}^{b}\omega_{+}^{\prime}(y)
W_{q}(y-x)\mathrm{d}y=\lim\limits_{x\uparrow b}\int_{x}^{b}\omega_{+}^{\prime}(y)
W_{q}(y-x)\mathrm{d}y=0$, we have that
\begin{eqnarray}
V_{0,b^{\omega}}^{\omega\prime}(b^{\omega}-)
\hspace{-0.3cm}&=&\hspace{-0.3cm}
1+\frac{W_{q}(0+)}{W_{q}(b^{\omega})}\bigg[-Z_{q}(b^{\omega})+\phi
+\lambda\int_{0}^{b^{\omega}}\omega_{+}^{\prime}(y)W_{q}(y)
\mathrm{d}y\bigg]=1=V_{0,b^{\omega}}^{\omega\prime}(b^{\omega}+),\nonumber
\end{eqnarray}
and
\begin{eqnarray}
V_{0,b^{\omega}}^{\omega\prime}(0+)
\hspace{-0.3cm}&=&\hspace{-0.3cm}
Z_{q}(b^{\omega})
-\lambda\int_{0}^{b^{\omega}}\omega_{+}^{\prime}(y)W_{q}(y-x)\mathrm{d}y=\phi,\nonumber
\end{eqnarray}
implying the continuous differentiability of $V_{0,b^{\omega}}^{\omega}(x)$ over $(-\infty,\infty)$. When $X$ has paths of unbounded variation, the scale function $W_{q}$ is continuously differentiable, and hence
\begin{eqnarray}
V_{0,b^{\omega}}^{\omega\prime\prime}(b^{\omega}+)-V_{0,b^{\omega}}^{\omega\prime\prime}(b^{\omega}-)
\hspace{-0.3cm}&=&\hspace{-0.3cm}
-\frac{W_{q}^{\prime}(0+)}{W_{q}(b^{\omega})}
\bigg[Z_{q}(b^{\omega})-\phi
-\lambda\int_{0}^{b^{\omega}}\omega_{+}^{\prime}(y)W_{q}(y)
\mathrm{d}y\bigg]
+q W_{q}(0+)
\nonumber\\
\hspace{-0.3cm}&&\hspace{-0.3cm}
-\lambda
\lim\limits_{x\uparrow b^{\omega}}\int_{0}^{b^{\omega}}\omega_{+}^{\prime}(y)
W_{q}^{\prime}(y-x)\mathrm{d}y
\nonumber\\
\hspace{-0.3cm}&=&\hspace{-0.3cm}0,\nonumber
\end{eqnarray}
where, in the second equality, we have used the fact that $\left|\int_{0}^{b^{\omega}}\omega_{+}^{\prime}(y)
W_{q}^{\prime}(y-x)\mathrm{d}y\right|\leq\omega_{+}^{\prime}(0^+)W_{q}(b^{\omega}-x) \rightarrow 0$ as $x\uparrow b^{\omega}$. As a result, $V_{0,b^{\omega}}^{\omega}(x)$ is twice continuously differentiable over $(0,\infty)$ when $X$ has paths of unbounded variation.
\end{proof}

Recall that $b^{\omega}\in (0,\infty)$ is the unique solution of \eqref{smoothcond.1}. The expected NPV $V_{0,b^{\omega}}^{\omega}(x)$ with the double barrier $(0,b^{\omega})$ can be reduced to
\begin{eqnarray}
\label{v0b*(x)}
V_{0,b^{\omega}}^{\omega}(x)
\hspace{-0.3cm}&=&\hspace{-0.3cm}
\left\{
\small
\begin{aligned}
&
-
\overline{Z}_{q}(b^{\omega}-x)
-\frac{\psi^{\prime}(0+)}{q}
+
\frac{\lambda}{q}\bigg[\omega(0)+
\int_{0}^{b^{\omega}}\omega_{+}^{\prime}(y)
Z_{q}(y-x)\mathrm{d}y\bigg],& x\in[0,b^{\omega}],\\
&x-b^{\omega}+V_{0,b^{\omega}}^{\omega}(b^{\omega}),& x\in(b^{\omega},\infty),\\
&\phi x+V_{0,b^{\omega}}^{\omega}(0),& x\in(-\infty,0).
\end{aligned}
\right.
\end{eqnarray}
%And, the condition $g(b)=0$ can be rewritten as
%\begin{eqnarray}
%\frac{\overline{Z}_{q}(b)
%-\overline{Z}_{q}(b-\zeta)+\frac{\lambda}{q}\Big[\omega(0)+\int_{0}^{b}\omega_{+}^{\prime}(y)Z_{q}(y-\zeta)\mathrm{d}y\Big]-\phi \zeta-K+\lambda\int_{0}^{b}\omega(y)W_{q}(y)\mathrm{d}y}{Z_{q}(b)}-\frac{\lambda}{q}\omega(b)=0,\end{eqnarray}which is equivalent to\begin{eqnarray}\label{zetabzero.1}\overline{Z}_{q}(b)-\overline{Z}_{q}(b-\zeta)-\frac{\lambda}{q}\int_{0}^{b}\omega_{+}^{\prime}(y)\left[Z_{q}(y)-Z_{q}(y-\zeta)\right]\mathrm{d}y-\phi\zeta-K=0.\end{eqnarray}
%We further assume that $V_{0,b}^{\omega\prime}(0)=\phi$, i.e.,

\begin{lem}\label{lem2.2}
The expected NPV $V_{0,b^{\omega}}^{\omega}(x)$ is increasing and concave over $(-\infty,\infty)$. In addition, we have $(V_{0,b^{\omega}}^{\omega})'(x)=\phi$ for $x\in(-\infty,0]$, and $(V_{0,b^{\omega}}^{\omega})'(x)=1$ for $x\in[b^{\omega},\infty)$.
\end{lem}

\begin{proof}
Define a process $\{V_{t}^{b^{\omega}};t\geq 0\}$ that
\begin{eqnarray}
\label{def.V}
V_{t}^{b^{\omega}}
%\hspace{-0.3cm}&=&\hspace{-0.3cm}
%X_{t}-\sup_{s\in[0,t]}(X_{s}-b^{\omega})\vee 0
%\nonumber\\
\hspace{-0.3cm}&:=&\hspace{-0.3cm}
%b^{\omega}-\Big[(
b^{\omega}-X_{t}-\inf_{0\leq s\leq t}(b^{\omega}-X_{s})\wedge 0,\quad t\geq 0,
%\Big],
\end{eqnarray}
which is the spectrally negative L\'evy process $\{b^{\omega}-X_{t};t\geq0\}$ reflected at its infimum.
In addition, denote
\begin{eqnarray}
\label{def.sigma}
\sigma_{b^{\omega}}^{+}:=\inf\{t\geq0; V_{t}^{b^{\omega}}\geq b^{\omega}\}=\inf\{t\geq 0; X_{t}-\sup_{0\leq s\leq t}(X_{s}-b^{\omega})\vee 0\leq 0\}.
\end{eqnarray}
Then, by Proposition 2 and Theorem 1 of \cite{Pistorius04}, we have that
\begin{eqnarray}
\label{Lap.sigma}
\mathrm{E}_{x}\left[e^{-q\sigma_{b^{\omega}}^{+}}\right]
\hspace{-0.3cm}&=&\hspace{-0.3cm}
\frac{Z_{q}(b^{\omega}-x)}{Z_{q}(b^{\omega})},
\\
\label{Pot.Y}
\int_{0}^{\infty}e^{-qt}\mathrm{P}_{x}\left(b^{\omega}-V_{t}^{b^{\omega}}\in\mathrm{d}y,\,t<\sigma_{b^{\omega}}^{+}\right)\mathrm{d}t
\hspace{-0.3cm}&=&\hspace{-0.3cm}
\bigg[\frac{Z_{q}(b^{\omega}-x)}{Z_{q}(b^{\omega})}W_{q}(y)-W_{q}(y-x)\bigg]\mathbf{1}_{[0,b^{\omega}]}(y)\mathrm{d}y.
\end{eqnarray}
By \eqref{v0b*(x)}, \eqref{Lap.sigma}, \eqref{Pot.Y} and the definition of $b^{\omega}$, we have that
\begin{eqnarray}
\label{V'0b*(x)}
\hspace{-0.3cm}
(V_{0,b^{\omega}}^{\omega})'(x)
\hspace{-0.3cm}&=&\hspace{-0.3cm}
Z_{q}(b^{\omega}-x)
-\lambda\int_{0}^{b^{\omega}}\omega_{+}^{\prime}(y)W_{q}(y-x)\mathrm{d}y
\nonumber\\
\hspace{-0.3cm}&=&\hspace{-0.3cm}
\bigg[\phi
+\lambda\int_{0}^{b^{\omega}}\omega_{+}^{\prime}(y)W_{q}(y)
\mathrm{d}y-Z_{q}(b^{\omega})\bigg]\frac{Z_{q}(b^{\omega}-x)}{Z_{q}(b^{\omega})}
+Z_{q}(b^{\omega}-x)
-\lambda\int_{0}^{b^{\omega}}\omega_{+}^{\prime}(y)W_{q}(y-x)\mathrm{d}y
\nonumber\\
\hspace{-0.3cm}&=&\hspace{-0.3cm}
\phi \frac{Z_{q}(b^{\omega}-x)}{Z_{q}(b^{\omega})}
+\lambda\int_{0}^{b^{\omega}}\omega_{+}^{\prime}(y)\left[
\frac{Z_{q}(b^{\omega}-x)}{Z_{q}(b^{\omega})}W_{q}(y)-W_{q}(y-x)\right]
\mathrm{d}y
\nonumber\\
\hspace{-0.3cm}&=&\hspace{-0.3cm}
\phi\mathrm{E}_{x}\left[e^{-q\sigma_{b^{\omega}}^{+}}\right]+\lambda
\int_{0}^{\infty}\omega_{+}^{\prime}(y)\int_{0}^{\infty}e^{-qt}\mathrm{P}_{x}\left(b^{\omega}-V_{t}^{b^{\omega}}\in\mathrm{d}y,t<\sigma_{b^{\omega}}^{+}\right)\mathrm{d}t
\nonumber\\
\hspace{-0.3cm}&=&\hspace{-0.3cm}
\phi-
\mathrm{E}_{x}\left[\int_{0}^{\sigma_{b^{\omega}}^{+}}qe^{-qt}\left(\phi-\frac{\lambda}{q}\omega_{+}^{\prime}(b^{\omega}-V_{t}^{b^{\omega}})\right)\mathrm{d}t\right].
\end{eqnarray}
By their definitions, we know that $V_{t}^{b^{\omega}}$ is non-increasing and $\sigma_{b^{\omega}}^{+}$ is non-decreasing with respect to the starting value $x$ of the process $X$, which combined with the concavity of $\omega$ results in the fact that the function $x\mapsto \phi-
\mathrm{E}_{x}\left[\int_{0}^{\sigma_{b^{\omega}}^{+}}qe^{-qt}\Big(\phi-\frac{\lambda}{q}\omega_{+}^{\prime}(b^{\omega}-V_{t}^{b^{\omega}})\Big)\mathrm{d}t\right]$ is non-increasing over $[0,b^{\omega}]$.
Hence, the desired result is verified.
\end{proof}

%We now propose a candidate optimal value function as\begin{eqnarray}\label{cand.opt.V}V(x)\hspace{-0.3cm}&=&\hspace{-0.3cm}\left\{\small\begin{aligned}&-\overline{Z}_{q}(b^{\omega}-x)-\frac{\psi^{\prime}(0+)}{q}+\frac{\lambda\omega(b^{\omega})}{q}Z_{q}(b^{\omega}-x)-\lambda\int_{0}^{b^{\omega}}\omega(y)W_{q}(y-x)\mathrm{d}y,& x\in(0,b^{\omega}],\\&x-b^{\omega}+V(b^{\omega}),& x\in(b^{\omega},\infty),\\&V(0)+\phi x,&x\in(-\infty,0].\end{aligned}\right.\end{eqnarray}By Lemma \ref{lem2.1} and Lemma \ref{lem2.2}, the candidate optimal value function $V(x)$ is continuously differentiable over $(-\infty,\infty)$. Actually, we have\begin{eqnarray}\label{V(x).prime}V^{\prime}(x)\hspace{-0.3cm}&=&\hspace{-0.3cm}\Big[Z_{q}(b^{\omega}-x)-\lambda\int_{0}^{b^{\omega}}\omega_{+}^{\prime}(y)W_{q}(y-x)\mathrm{d}y\Big]\mathbf{1}_{0\leq x\leq b^{\omega}}+\mathbf{1}_{x>b^{\omega}}+\phi\,\mathbf{1}_{x<0},\end{eqnarray}which, by Lemma \ref{lem2.2}, is concave over $(-\infty,\infty)$.Furthermore, if the paths of $X$ have unbounded variation, $V(x)$ is twice continuously differentiable over $(0,\infty)$ such that\begin{eqnarray}\label{V(x).2prime}V^{\prime\prime}(x)\hspace{-0.3cm}&=&\hspace{-0.3cm}\Big[-q W_{q}(b^{\omega}-x)+\lambda\int_{0}^{b^{\omega}}\omega_{+}^{\prime}(y)W_{q}^{\prime}(y-x)\mathrm{d}y\Big]\mathbf{1}_{0< x\leq b^{\omega}}.\end{eqnarray}

In our previous conjecture, the optimality of the control problem \eqref{valfun} can be attained by a double barrier dividend and capital injection strategy. To verify the optimality among all admissible dividend and capital injection singular controls, let us first characterize the optimal one among all admissible double barrier dividend and capital injection strategies. We then proceed to verify that the obtained optimal double barrier strategy is also the optimal control attaining the value function \eqref{valfun} among all admissible singular controls. With this two-step procedure in mind, we first establish the next result, which states that the double barrier strategy with the couple of barriers $(0,b^{\omega})$ is the optimal one among all double barrier dividend and capital injection strategies with barriers $(0,b)$ such that $b\in(0,\infty$), i.e., the value function $V_{0,b^{\omega}}^{\omega}(x)$ dominates the value function $V_{0,b}^{\omega}(x)$ for all $x\in(-\infty, \infty)$ and $b\neq b^{\omega}$ with $b\in(0,\infty)$.

\begin{lem}
\label{lem2.3}
Fix $b\neq b^{\omega}$. Recall that the expected NPV $V^{\omega}_{0,b}(x)$ is given by \eqref{v0b(x)}.
Let us consider
$g(x):=V_{0,b^{\omega}}^{\omega}(x)-V^{\omega}_{0,b}(x)$.
Then, $g(x)$ is non-decreasing over $(-\infty,\infty)$ and satisfies that $g(x)\geq 0$ for all $x\in(-\infty, \infty)$.
\end{lem}

\begin{proof}
We only provide the proof for the case $b\in(b^{\omega},\infty)$ because the proof for the case $b\in(0,b^{\omega})$ is very similar.
By the fluctuation identity \eqref{fluc.2}, we know that
\begin{eqnarray}
\label{ineq.pot.1}
Z_{q}(y)-\frac{Z_{q}(b)}{W_{q}(b)}W_{q}(y)\geq 0, \quad y\in[0,b].
\end{eqnarray}
By \eqref{ineq.pot.1} and the facts that $\omega_{+}^{\prime}(y)<q/\lambda$ for $y\geq b^{\omega}$, $Z_{q}(b)-\phi
-\lambda\int_{0}^{b}\omega_{+}^{\prime}(y)W_{q}(y)
\mathrm{d}y>0$ for $b>b^{\omega}$, and $Z_{q}(b^{\omega})-\phi
-\lambda\int_{0}^{b^{\omega}}\omega_{+}^{\prime}(y)W_{q}(y)
\mathrm{d}y=0$ (see Lemma \ref{lem2.1}), we have
\begin{eqnarray}\label{g0}
g(0)
\hspace{-0.3cm}&=&\hspace{-0.3cm}
-
\overline{Z}_{q}(b^{\omega})
-\frac{\psi^{\prime}(0+)}{q}
+
\frac{\lambda}{q}\bigg[\omega(0)+
\int_{0}^{b^{\omega}}\omega_{+}^{\prime}(y)
Z_{q}(y)\mathrm{d}y\bigg]
\nonumber\\
\hspace{-0.3cm}&&\hspace{-0.3cm}
-\left[-\overline{Z}_{q}(b)-\frac{\psi^{\prime}(0+)}{q}+
\frac{\lambda}{q}\bigg[\omega(0)+
\int_{0}^{b}\omega_{+}^{\prime}(y)
Z_{q}(y)\mathrm{d}y\bigg]\right]
\nonumber\\
\hspace{-0.3cm}&&\hspace{-0.3cm}
+Z_{q}(b)\bigg[\frac{Z_{q}(b^{\omega})-\phi
-\lambda\int_{0}^{b^{\omega}}\omega_{+}^{\prime}(y)W_{q}(y)
\mathrm{d}y}{qW_{q}(b)}
-\frac{Z_{q}(b)-\phi
-\lambda\int_{0}^{b}\omega_{+}^{\prime}(y)W_{q}(y)
\mathrm{d}y}{qW_{q}(b)}\bigg]
\nonumber\\
\hspace{-0.3cm}&=&\hspace{-0.3cm}
\int_{b^{\omega}}^{b}\left(1-\frac{\lambda}{q}\omega_{+}^{\prime}(y)\right)
Z_{q}(y)\mathrm{d}y
-\frac{Z_{q}(b)}{W_{q}(b)}
\int_{b^{\omega}}^{b}\left(1-\frac{\lambda}{q}\omega_{+}^{\prime}(y)\right)W_{q}(y)\mathrm{d}y
\nonumber\\
\hspace{-0.3cm}&=&\hspace{-0.3cm}
\int_{b^{\omega}}^{b}\left(1-\frac{\lambda}{q}\omega_{+}^{\prime}(y)\right)
\left[Z_{q}(y)
-\frac{Z_{q}(b)}{W_{q}(b)}
W_{q}(y)\right]\mathrm{d}y
\nonumber\\
\hspace{-0.3cm}&\geq &\hspace{-0.3cm}0.
\end{eqnarray}
Using Theorem 8.7 in Chapter 8 of \cite{Kyp2014}, we arrive at
\begin{eqnarray}
\int_{0}^{\infty}e^{-qt}\mathrm{P}_{x}\left(b-X_{t}\in\mathrm{d}y,t<\tau_{0}^{-}\wedge\tau_{b}^{+}\right)\mathrm{d}t=\bigg[\frac{W_{q}(b-x)}{W_{q}(b)}W_{q}(y)-W_{q}(y-x)\bigg]\mathbf{1}_{[0,b]}(y)\mathrm{d}y,\nonumber
\end{eqnarray}
which implies that
\begin{eqnarray}
\label{ineq.pot.2}
\frac{W_{q}(b-x)}{W_{q}(b)}W_{q}(y)-W_{q}(y-x)\geq 0,\quad y\in[0,b].
\end{eqnarray}
By Lemma \ref{lem2.1}, \eqref{ineq.pot.2}, and the definitions of $V_{0,b^{\omega}}^{\omega}$ and $V^{\omega}_{0,b}$, one can check that
\begin{eqnarray}\label{g'x.1}
g^{\prime}(x)\hspace{-0.3cm}&=&\hspace{-0.3cm}
1
-\Big[Z_{q}(b-x)-\lambda\int_{0}^{b}\omega_{+}^{\prime}(y)W_{q}(y-x)\mathrm{d}y\Big]
\nonumber\\
\hspace{-0.3cm}&&\hspace{-0.3cm}
+\frac{W_{q}(b-x)}{W_{q}(b)}\Big[Z_{q}(b)-\phi
-\lambda\int_{0}^{b}\omega_{+}^{\prime}(y)W_{q}(y)
\mathrm{d}y\Big]
\nonumber\\
\hspace{-0.3cm}&&\hspace{-0.3cm}
-\frac{W_{q}(b-x)}{W_{q}(b)}\Big[Z_{q}(b^{\omega})-\phi
-\lambda\int_{0}^{b^{\omega}}\omega_{+}^{\prime}(y)W_{q}(y)
\mathrm{d}y\Big]
\nonumber\\
\hspace{-0.3cm}&=&\hspace{-0.3cm}
\int_{x}^{b}\left(-q+\lambda\omega_{+}^{\prime}(y)\right)W_{q}(y-x)\mathrm{d}y
\nonumber\\
\hspace{-0.3cm}&&\hspace{-0.3cm}
+\frac{W_{q}(b-x)}{W_{q}(b)}\bigg[\int_{x}^{b}\left(q-\lambda\omega_{+}^{\prime}(y)\right)W_{q}(y)\mathrm{d}y
+\int_{b^{\omega}}^{x}\left(q-\lambda\omega_{+}^{\prime}(y)\right)W_{q}(y)\mathrm{d}y\bigg]
\nonumber\\
\hspace{-0.3cm}&=&\hspace{-0.3cm}
\frac{W_{q}(b-x)}{W_{q}(b)}\int_{b^{\omega}}^{x}\left(q-\lambda\omega_{+}^{\prime}(y)\right)W_{q}(y)\mathrm{d}y
\nonumber\\
\hspace{-0.3cm}&&\hspace{-0.3cm}
+\int_{x}^{b}\left(q-\lambda\omega_{+}^{\prime}(y)\right)\left[\frac{W_{q}(b-x)}{W_{q}(b)}W_{q}(y)-W_{q}(y-x)\right]\mathrm{d}y
\nonumber\\
\hspace{-0.3cm}&\geq &\hspace{-0.3cm}
0,\quad x\in[b^{\omega},b),
\end{eqnarray}
and
\begin{eqnarray}\label{g'x.2}
g^{\prime}(x)\hspace{-0.3cm}&=&\hspace{-0.3cm}
Z_{q}(b^{\omega}-x)
-
\lambda
\int_{0}^{b^{\omega}}\omega_{+}^{\prime}(y)
W_{q}(y-x)\mathrm{d}y
-\Big[Z_{q}(b-x)-\lambda\int_{0}^{b}\omega_{+}^{\prime}(y)W_{q}(y-x)\mathrm{d}y\Big]
\nonumber\\
\hspace{-0.3cm}&&\hspace{-0.3cm}
+\frac{W_{q}(b-x)}{W_{q}(b)}\Big[Z_{q}(b)-\phi
-\lambda\int_{0}^{b}\omega_{+}^{\prime}(y)W_{q}(y)
\mathrm{d}y\Big]
\nonumber\\
\hspace{-0.3cm}&&\hspace{-0.3cm}
-\frac{W_{q}(b-x)}{W_{q}(b)}\Big[Z_{q}(b^{\omega})-\phi
-\lambda\int_{0}^{b^{\omega}}\omega_{+}^{\prime}(y)W_{q}(y)
\mathrm{d}y\Big]
\nonumber\\
\hspace{-0.3cm}&=&\hspace{-0.3cm}
\int_{b^{\omega}}^{b}\left(-q+\lambda\omega_{+}^{\prime}(y)\right)W_{q}(y-x)\mathrm{d}y
+\frac{W_{q}(b-x)}{W_{q}(b)} \int_{b^{\omega}}^{b}\left(q-\lambda\omega_{+}^{\prime}(y)\right)W_{q}(y)\mathrm{d}y
\nonumber\\
\hspace{-0.3cm}&=&\hspace{-0.3cm}
\int_{b^{\omega}}^{b}\left(q-\lambda\omega_{+}^{\prime}(y)\right)
\left[\frac{W_{q}(b-x)}{W_{q}(b)}W_{q}(y)-W_{q}(y-x)\right]\mathrm{d}y
\nonumber\\
\hspace{-0.3cm}&\geq &\hspace{-0.3cm}
0,\quad x\in[0,b^{\omega}).
\end{eqnarray}
Putting \eqref{g0}, \eqref{g'x.1}, \eqref{g'x.2}, and the fact that $g^{\prime}(x)\equiv0$ over $(-\infty,0)\cup(b,\infty)$ all together, we obtain the desired result.
\end{proof}

For any function that is sufficiently differentiable, let us define an operator $\mathcal{A}$ on $f$ that
$$
\mathcal{A}f(x):=\frac{1}{2}\sigma^{2}f^{\prime\prime}(x)
-\gamma f^{\prime}(x)+\int_{(0,\infty)}\left(f(x+y)-f(x)-f^{\prime}(x) y\mathbf{1}_{(0,1)}(y)\right)\upsilon(\mathrm{d}y),
$$
where $x\in(-\infty,\infty)$. The next verification lemma establishes the connection between the associated HJB variational inequality and the value function of the auxiliary control problem \eqref{valfun}. In particular, it enables us to compare the value function $V_{0,b^{\omega}}^{\omega}(x)$ under the candidate optimal double barrier $(0, b^{\omega})$ and the expected NPV $V_{(D,R)}^{\omega}(x)$ under any admissible singular controls $(D, R)$.

\begin{lem}[Verification Lemma]\label{lem2.7}
Suppose that %$\int_{1}^{\infty}y\upsilon(\mathrm{d}y)<\infty$ and
the function $f(x)$ is non-decreasing and continuously differentiable over $(-\infty,\infty)$.
Furthermore, suppose that $f(x)$ is twice continuously differentiable over $(0,\infty)$ when $X$ has paths of unbounded variation, and
that
\begin{eqnarray}
\label{HJB}
\max\{\left(\mathcal{A}-q\right)f(x){\color{black}+\lambda\omega(x)}, 1-f^{\prime}(x), f^{\prime}(x)-\phi\}\leq 0.
\end{eqnarray}
%\eqref{unif.boun.s.d.f.0}, \eqref{unif.boun.s.d.f.1}, \eqref{Optimality conditionplus.0.0},  and \eqref{Optimality condition.0},
%{\color{black}
%\begin{eqnarray}\label{D_reflect}
%R^*_{t}=-\inf\limits_{0\leq s\leq t}\left(X_{s}-D^*_{s}\right)\wedge 0,
%\nonumber
%\end{eqnarray}
%and}
Then
%$(D^*, R^*)$ is the optimal strategy, and
$f(x)\geq  V_{(D,R)}^{\omega}(x)$ for all $x\in\mathbb{R}$ and all admissible dividend and capital injection singular controls $(D,R)$.
 \end{lem}

\begin{proof}
Let $\mathcal{D}$ be the set of admissible dividend and capital injection strategy $(D_{t},R_{t})_{t\geq0}$ with $R_{t}$ being continuous and of form \eqref{R.form}.
By Lemma \ref{R.continu.}, we only need to prove that $f(x)$ dominates the value function of any admissible dividend and capital injection strategies among $\mathcal{D}$.
For a given strategy $(D,R)\in\mathcal{D}$, recall that $U_{t}=X_{t}-D_{t}+R_{t}$ for $t\geq0$. %{\color{black} We follow \cite{Loe2009a} to define $ {D}_{t}$ as the c\`{a}dl\`{a}g version of $D_{t}$ and $ {U}_{t}:=X_{t}- {D}_{t}+R_{t}$; and
We follow Theorem 2.1 in \cite{Kyp2014} to denote $X_{t}$ as the sum of the independent processes $-\gamma t+\sigma B_{t}$, $\sum_{s\leq t}\Delta X_{s}\mathbf{1}_{\{\Delta X_{s}\geq 1\}}$, and
%\begin{eqnarray}
$X_{t}+\gamma t-\sigma B_{t}-\sum_{s\leq t}\Delta X_{s}\mathbf{1}_{\{\Delta X_{s}\geq 1\}}$,
%\nonumber\end{eqnarray}
with the latter one being a square integrable martingale.
%X_{t}-\gamma t+\sigma B_{t}$.
%It is seen that the four processes $ {U}$, $X$, $ {D}$, and $R$ are all  c\`{a}dl\`{a}g and adapted stochastic processes. %In addition, by the L\'evy-It\^{o} decomposition theorem (see, Theorem 2.1 in \cite{Kyprianou06}), we have the following representation
%$$X_{t}=\gamma t+\sigma B_{t}+\sum_{s\leq t}\Delta X_{s}\mathbf{1}_{\{\Delta X_{s}\geq 1\}}.$$}
Denote by $\{ {U}^{c}_{t};t\geq0\}$ and $\{ {D}^{c}_{t};t\geq0\}$
as the continuous part of  $\{ {U}_{t};t\geq0\}$ and $\{ {D}_{t};t\geq0\}$, respectively. %{\color{blue}Let, for positive integer $N\geq 1$
%$$T_{N}=\inf\{t\geq0;  {U}_{t}>N\}$$
%be the sequence of localization stopping times. Then, due to the definitions of $\mathcal{D}$ and $T_{N}$,  it holds that for all $t<T_{N}$
%$ {U}_{t}\geq d+\tfrac{1}{N}$ or $ {U}_{t}\leq d$, and
%\begin{equation}\label{U.bounded.up.down}
%0\leq  {U}_{t}\leq N,
%\end{equation}
%i.e., both $ {U}(t-)$ and $ {U}_{t}$ are restricted to the bounded compact set $\left[0,N\right]$. X_{t}-\gamma t-\sigma B_{t}-\sum_{s\leq t}\Delta X_{s}\mathbf{1}_{\{\Delta X_{s}\geq 1\}}
%Let $f(x):=V_{(D^*,R^*)}(x)$. Hence, $f$ is twice differentiable over $[0,\infty)$, and satisfies \eqref{Optimality conditionplus.0.0} and \eqref{Optimality condition.0r}.
By Theorem 4.57 (It\^{o}'s formula) in \cite{JaSh2003}, we have, for $x\in(0,\infty)$,
{\color{black}\begin{eqnarray}\label{2.2}
\hspace{-0.3cm}&&\hspace{-0.3cm}
e^{-q t}f( {U}_{t})\nonumber\\
\hspace{-0.3cm}&=&\hspace{-0.3cm}
f(x)-\int_{0-}^{t}q e^{-q s}f( {U}_{s-})\mathrm{d}s+\int_{0-}^{t} e^{-q s}f^{\prime}( {U}_{s-})\mathrm{d} {U}_{s}\nonumber\\
\hspace{-0.3cm}&&\hspace{-0.3cm}
+\frac{1}{2} \int_{0-}^{t} e^{-q s}f^{\prime\prime}( {U}_{s-})\mathrm{d}\langle  {U}^{c}(\cdot), {U}^{c}(\cdot)\rangle_{s}
\nonumber\\
\hspace{-0.3cm}&&\hspace{-0.3cm}
+\sum_{s\leq t}e^{-q s}\big(f( {U}_{s-}+\Delta  {U}_{s})-f( {U}_{s-})-f^{\prime}( {U}_{s-})\Delta  {U}_{s}\big)
%\nonumber\\
%&&+\sum_{s\leq t}e^{-q s}\left(f( {U}(s+))-f\left( {U}_{s}\right)+f^{\prime}( {U}_{s})\Delta  {D}_{s}\right)
\nonumber\\
\hspace{-0.3cm}&=&\hspace{-0.3cm}
f(x)-\int_{0-}^{t}q e^{-q s}f( {U}_{s-})\mathrm{d}s+\int_{0-}^{t} e^{-q s}f^{\prime}( {U}_{s-})
\mathrm{d}(-\gamma s+\sigma B_{s})\nonumber\\
\hspace{-0.3cm}&&\hspace{-0.3cm}
+\int_{0-}^{t} e^{-q s}f^{\prime}( {U}_{s-})
\mathrm{d}\big(X_{s}+\gamma s-\sigma B_{s}-\sum_{r\leq s}\Delta X_r\mathbf{1}_{\{\Delta X_r\geq 1\}}\big)
\nonumber\\
\hspace{-0.3cm}&&\hspace{-0.3cm}
+\int_{0-}^{t} e^{-q s}f^{\prime}( {U}_{s-})
\mathrm{d}\big(R_{s}- {D}^{c}_s-\sum_{r\leq s}\Delta  {D}_r\big)
\nonumber\\
\hspace{-0.3cm}&&\hspace{-0.3cm}
+\int_{0-}^{t} e^{-q s}f^{\prime}( {U}_{s-})
\mathrm{d}\big(\sum_{r\leq s}\Delta X_r\mathbf{1}_{\{\Delta X_r\geq 1\}}\big)
+\frac{\sigma^{2}}{2} \int_{0-}^{t} e^{-q s}f^{\prime\prime}( {U}_{s-})\mathrm{d}s
\nonumber\\
\hspace{-0.3cm}&&\hspace{-0.3cm}
+\sum_{s\leq t}e^{-q s}\big[f( {U}_{s-}+\Delta X_{s})-f( {U}_{s-})-f^{\prime}( {U}_{s-})\Delta X_{s}\big]
%\nonumber\\
%&&+\sum_{s\leq t}e^{-q s}\big[f( {U}_{s-}+\Delta X_{s})
%\nonumber\\
%&&\quad\quad\quad\quad\quad\,\,
%-f( {U}_{s-}+\Delta X_{s})-f^{\prime}( {U}_{s-})R\big]
\nonumber\\
\hspace{-0.3cm}&&\hspace{-0.3cm}
+\sum_{s\leq t}e^{-q s}\big[f( {U}_{s-}+\Delta  {U}_{s})-f( {U}_{s-}+\Delta X_{s})
%\nonumber\\
%&&\quad\quad\quad\quad\quad\,\,
+f^{\prime}( {U}_{s-})\Delta  {D}_{s}\big]
\nonumber\\
\hspace{-0.3cm}&=&\hspace{-0.3cm}
f(x)-\int_{0-}^{t}q e^{-q s}f( {U}_{s-})\mathrm{d}s+\int_{0-}^{t} e^{-q s}f^{\prime}( {U}_{s-})
\mathrm{d}(-\gamma s+\sigma B_{s})\nonumber\\
\hspace{-0.3cm}&&\hspace{-0.3cm}
+\int_{0-}^{t} e^{-q s}f^{\prime}( {U}_{s-})
\mathrm{d}\big(X_{s}+\gamma s-\sigma B_{s}-\sum_{r\leq s}\Delta X_r\mathbf{1}_{\{\Delta X_r\geq 1\}}\big)
\nonumber\\
\hspace{-0.3cm}&&\hspace{-0.3cm}
+\int_{0-}^{t} e^{-q s}f^{\prime}( {U}_{s-})
\mathrm{d}(R_{s}- {D}^{c}_{s})
+\frac{\sigma^{2}}{2} \int_{0-}^{t} e^{-q s}f^{\prime\prime}( {U}_{s-})\mathrm{d}s
\nonumber\\
\hspace{-0.3cm}&&\hspace{-0.3cm}
+\sum_{s\leq t}e^{-q s}\big[f( {U}_{s-}+\Delta X_{s})-f( {U}_{s-})
%\nonumber\\
%&&\quad \quad \quad \quad \quad \,\,
-f^{\prime}( {U}_{s-})\Delta X_{s}\mathbf{1}_{\{\Delta X_{s}<1\}}\big]
%\nonumber\\
%&&+\sum_{s\leq t}e^{-q s}\big[f( {U}_{s-}+\Delta X_{s}+\Delta R_{s})
%-f( {U}_{s-}+\Delta X_{s})\big]
\nonumber\\
\hspace{-0.3cm}&&\hspace{-0.3cm}
+\sum_{s\leq t}e^{-q s}\big[f( {U}_{s-}+\Delta  {U}_{s})-f( {U}_{s-}+\Delta X_{s})\big],
\end{eqnarray}}where $\Delta  {D}_{s}= {D}_{s}- {D}_{s-}$, $\Delta X_{s}=X_{s}-X_{s-}$, and, $\Delta  {U}_{s}= {U}_{s}- {U}_{s-}=\Delta X_{s}-\Delta  {D}_{s}$. %Due to the fact that $(D,R)\in\mathcal{D}$, one knows that $\Delta R_{s}>0$ implies a jump of $N(\cdot,\cdot)$ at time $s$ (i.e., whenever there is a jump in $R$, there must be a jump in $X$).
By the fact that $V^{\prime}(x)\geq 1$ for all $x\in[0,\infty)$ (see; \eqref{HJB}), %and the fact that  $\Delta  {D}_{s}>c$ whenever $\Delta  {D}_{s}>0$,
we have, for $s\in\left[0,t\right)$,
\begin{eqnarray}
&&\hspace{-0.6cm}f( {U}_{s-}+\Delta  {U}_{s})-f( {U}_{s-}+\Delta X_{s})+ \Delta  {D}_{s}\le 0.\label{2.3}
%\\&&f( {U}_{s-}+\Delta X_{s}+\Delta R_{s})-f( {U}_{s-}+\Delta X_{s})\leq \phi\Delta R_{s}.\label{f'phi}
\end{eqnarray}
Therefore, by \eqref{HJB}, %\eqref{Optimality condition.0r},
(\ref{2.2}) and (\ref{2.3}), %and (\ref{f'phi}),
we have
{\color{black}\begin{eqnarray}\label{3.10}
\hspace{-0.3cm}&&\hspace{-0.3cm}
e^{-q t}f( {U}_{t})
\nonumber\\
\hspace{-0.3cm}&=&\hspace{-0.3cm}
f(x)+\int_{0-}^{t}e^{-q s}(\mathcal{A}-q)f( {U}_{s-})\mathrm{d}s
+\int_{0-}^{t}\sigma e^{-q s}f^{\prime}( {U}_{s-})\mathrm{d}B_{s}
\nonumber\\
\hspace{-0.3cm}&&\hspace{-0.3cm}
+\int_{0-}^{t} e^{-q s}f^{\prime}( {U}_{s-})
\mathrm{d}\big(X_{s}+\gamma s-\sigma B_{s}-\sum_{r\leq s}\Delta X_r\mathbf{1}_{\{\Delta X_r\geq 1\}}\big)
\nonumber\\
\hspace{-0.3cm}&&\hspace{-0.3cm}%{\color{black}-\int_{0-}^{t}\int_{0}^{1}e^{-q s}f^{\prime}( {U}_{s-})y\,\overline{N}(\mathrm{d}s,\mathrm{d}y)}
+\int_{0-}^{t} e^{-q s}f^{\prime}( {U}_{s-})\mathrm{d}(R_{s}- {D}^{c}_{s})
+\int_{0-}^{t}\int_{0}^{\infty}e^{-q s}\big[f( {U}_{s-}
+y)-f( {U}_{s-})
\nonumber\\
\hspace{-0.3cm}&&\hspace{-0.3cm}
-f^{\prime}( {U}_{s-})y
\mathbf{1}_{(0,1)}(y)\big]
\overline{N}(\mathrm{d}s,\mathrm{d}y)
%\nonumber\\
%&&+\sum_{s\leq t}e^{-q s}\big[f( {U}_{s-}+\Delta X_{s}+\Delta R_{s})
%-f( {U}_{s-}+\Delta X_{s})\big]
\nonumber\\
\hspace{-0.3cm}&&\hspace{-0.3cm}
+\sum_{s\leq t}e^{-q s}\big[f( {U}_{s-}+\Delta  {U}_{s})-f( {U}_{s-}+\Delta X_{s})\big]
\nonumber\\
\hspace{-0.3cm}&\leq&\hspace{-0.3cm} f(x){\color{black}-\lambda\int_{0}^{t}e^{-q s}\omega( {U}_{s-})\mathrm{d}s}+\phi\int_{0-}^{t} e^{-q s}\mathrm{d}R_{s}
-\int_{0-}^{t} e^{-q s}\mathrm{d} {D}^{c}_{s}
%+\phi\sum_{s\leq t}e^{-q s}\Delta R_{s}
+\int_{0-}^{t}\sigma e^{-q s}f^{\prime}( {U}_{s-})\mathrm{d}B_{s}
\nonumber\\
\hspace{-0.3cm}&&\hspace{-0.3cm}
+\int_{0}^{t} e^{-q s}f^{\prime}( {U}_{s-})
\mathrm{d}\big(X_{s}+\gamma s-\sigma B_{s}-\sum_{r\leq s}\Delta X_r\mathbf{1}_{\{\Delta X_r\geq 1\}}\big)
\nonumber\\
\hspace{-0.3cm}&&\hspace{-0.3cm}
%{\color{black}-\int_{0-}^{t}\int_{0}^{1}e^{-q s}f^{\prime}( {U}_{s-})y\,\overline{N}(\mathrm{d}s,\mathrm{d}y)}
 -\sum\limits_{s\leq t}e^{-q s} \Delta  {D}_{s}+\int_{0-}^{t}\int_{0}^{\infty}e^{-q s}
\big(f( {U}_{s-}+y)-f( {U}_{s-})
\nonumber\\
\hspace{-0.3cm}&&\hspace{-0.3cm}
-f^{\prime}( {U}_{s-})y
\mathbf{1}_{(0,1]}(y)\big)
\overline{N}(\mathrm{d}s,\mathrm{d}y)
,\quad x\in(0,\infty).
\end{eqnarray}}Define a sequence of stopping times $(T_{m})_{m\geq1}$ that %$T_{n}$ be the smallest time $t$ such that the following sum of integrals
\begin{eqnarray}
\hspace{-0.3cm}&&\hspace{-0.3cm}
T_{m}:=m\wedge\inf\{t\geq0; %\int_{0-}^{t}\int_{0}^{\infty}e^{-2J_{s}}\Big[ {V}( {U}_{s-}
%-y)- {V}( {U}_{s-})\Big]^{2}\mathbf{1}_{\overline{\mathbb{S}}}_{s}
%\,\mathrm{d}s\, {\upsilon}(\mathrm{d}y)
%\nonumber\\
%\hspace{-0.3cm}&&\hspace{-0.3cm}\quad\quad\quad\quad\quad\quad
%+
%\int_{0-}^{t} e^{-2J_{s}}\Big[\sigma^{2}
%\left[V^{\prime}( {U}_{s})\right]^{2}\mathbf{1}_{\mathbb{S}}_{s}
%+ {\sigma}^{2}
%\big[ {V}^{\prime}( {U}_{s})\big]^{2}\mathbf{1}_{\overline{\mathbb{S}}}_{s}
%\Big]\,\mathrm{d}s
%\nonumber\\
%\hspace{-0.3cm}&&\hspace{-0.3cm}\quad\quad\quad\quad\quad\quad
%+\int_{0-}^{t}\int_{0}^{\infty}e^{-2J_{s}}\Big[V( {U}_{s-}
%-y)-V( {U}_{s-})\Big]^{2}
%\mathbf{1}_{\mathbb{S}}_{s}
%\,\mathrm{d}s\,\upsilon(\mathrm{d}y)
 {U}_{t}\geq m\}, \quad m\geq 1.\nonumber
\end{eqnarray}
It follows that $T_{m}\rightarrow\infty$ almost surely as $m\rightarrow\infty$. In addition, $ {U}_{t-}$ is confined in the compact set $\left[0,m\right]$ for $t\leq T_{m}$.
By the L\'evy-It\^{o} decomposition theorem (see, Theorem 2.1 in \cite{Kyp2014}) or Appendix A in \cite{Loe09}, the stochastic integral
%s with respect to the compensated Poisson random measure
$${\color{black}\int_{0-}^{t\wedge T_{m}} e^{-q s}f^{\prime}( {U}_{s-})
\mathrm{d}\bigg(X_{s}+\gamma s-\sigma B_{s}-\sum_{r\leq s}\Delta X_r\mathbf{1}_{\{\Delta X_r\geq 1\}}\bigg)}, \quad t\geq0,$$
is a martingale starting from zero. % (see also, Appendix A of \cite{Loeffen09a}).
By Corollary 4.6 in \cite{Kyp2014} {\color{black}and the facts that $\int_{0}^{1}y^{2}\upsilon(\mathrm{d}y)<\infty$ (because $\upsilon$ is a L\'evy measure) and $\int_{1}^{\infty}y\upsilon(\mathrm{d}y)<\infty$ (by the assumption that $\mathrm{E}[X_1]<\infty$)},
%or Page 62 of \cite{Ikeda1981},
the following stochastic integral with respect to the  compensated Poisson random measure
\begin{eqnarray}
&&\int_{0-}^{t\wedge T_{m}}\int_{0}^{\infty}e^{-q s}\left(f( {U}_{s-}
+y)-f( {U}_{s-})
%\nonumber\\
%&&\quad \quad \quad \quad \quad
-f^{\prime}( {U}_{s-})y
\mathbf{1}_{(0,1]}(y)\right)
\overline{N}(\mathrm{d}s,\mathrm{d}y),\quad t\geq0,\nonumber
\end{eqnarray}
is a martingale starting from zero.
%is an $(\mathcal{F}_{t})$-martingale with zero mean. Indeed, the integrand of the above stochastic integration is bounded from below and above owing to (\ref{U.bounded.up.down}) and
%\begin{eqnarray*}
%\sup_{n\geq1}\sup_{x\in\left[0,N\right]}\left|f(x)\right|
%\leq\sup_{w\in\left[-2,N\right]}\left|f\left(w\right)\right|<\infty,
%\end{eqnarray*}
%where we have used (\ref{Def.of.fn.0}) and the fact that $f(x)\in C(-\infty,\infty)$.
Similarly, the following stochastic integral (see, Page 146 in \cite{KarS1991})
$$\int_{0-}^{t\wedge T_{m}}\sigma e^{-q s}f^{\prime}( {U}_{s-})\mathrm{d}B_{s},\quad t\geq 0,$$
is a martingale starting from zero.

Taking expectations on both sides of \eqref{3.10} after localization by $T_{m}$, we have
\begin{eqnarray}\label{45.v4}
f(x)
\hspace{-0.3cm}&\geq&\hspace{-0.3cm}
\mathrm{E}_x\left[e^{-q (t\wedge T_{m})}f( {U}_{t\wedge T_{m}})\right]
-\phi \mathrm{E}_x\Big[\int_{0-}^{t\wedge T_{m}} e^{-q s}\mathrm{d}R_{s}\Big]\nonumber\\
\hspace{-0.3cm}&&\hspace{-0.3cm}
+ \mathrm{E}_x\Big[\sum_{s\leq t\wedge T_{m}}e^{-q s}\Delta  {D}_{s}+\int_0^{t\wedge T_{m}}e^{-q s}\mathrm{d} {D}^{c}_{s}\Big]{\color{black}+\lambda\mathrm{E}_x\Big[\int_{0}^{t\wedge T_{m}}e^{-q s}\omega( {U}_{s-})\mathrm{d}s\Big]}
\nonumber\\
\hspace{-0.3cm}&\geq&\hspace{-0.3cm}\mathrm{E}_x\left[e^{-q (t\wedge T_{m})}{\color{black}f(0)}\right]
-\phi \mathrm{E}_x\Big[\int_{0-}^{t\wedge T_{m}} e^{-q s}\mathrm{d}R_{s}\Big]\nonumber\\
\hspace{-0.3cm}&&\hspace{-0.3cm}
+ \mathrm{E}_x\Big[\int_{0-}^{t\wedge T_{m}}e^{-q s}\mathrm{d} {D}_{s}\Big]
{\color{black}+\lambda\mathrm{E}_x\Big[\int_{0}^{t\wedge T_{m}}e^{-q s}\omega( {U}_{s})\mathrm{d}s\Big]},\quad x\in(0,\infty).
\end{eqnarray}
%where we have used the fact that, by \eqref{Def.of.fn.0} as well as the non-decreasing property of $f$ and $f$
%\begin{eqnarray}f( {U}(t\wedge T_{m}))
%&\geq& f(0), \quad n\geq 1.\nonumber
%\end{eqnarray}
By setting $n, t, m\to\infty$ in (\ref{45.v4}), and then taking use of the bounded convergence theorem (note that {\color{black}$f(0)$} is bounded), we get
\begin{eqnarray}\label{46.v4}
f(x)
\hspace{-0.3cm}&\geq&\hspace{-0.3cm}
-\phi \mathrm{E}_x\Big[\int_{0-}^{\infty} e^{-q s}\mathrm{d}R_{s}\Big]+ \mathrm{E}_x\Big[\int_{0-}^{\infty}e^{-q s}\mathrm{d} {D}_{s}\Big]{\color{black}+\lambda\mathrm{E}_x\Big[\int_{0}^{\infty}e^{-q s}\omega( {U}_{s})\mathrm{d}s\Big]}
%\nonumber\\\hspace{-0.3cm}&=&\hspace{-0.3cm}-\phi \mathrm{E}_x\Big(\int_{0-}^{\infty} e^{-q s}\mathrm{d}R_{s}\Big)+ \mathrm{E}_x\Big(\int_0^{\infty}e^{-q s}\mathrm{d}D_{s}\Big){\color{black}+\lambda\mathrm{E}_x\Big(\int_{0}^{\infty}e^{-q s}\omega(U_{s})\mathrm{d}s\Big)}
\nonumber\\
\hspace{-0.3cm}&=&\hspace{-0.3cm}V_{(D,R)}^{\omega}(x),\quad x\in(0,\infty).\nonumber
\end{eqnarray}
The arbitrariness of $(D,R)$ and the continuity of $f$ imply that
$f(x)\geq V_{(D,R)}^{\omega}(x)$ for all $x\in[0,\infty)$ and all admissible $(D,R)\in\mathcal{D}$. The reverse inequality is trivial, and the proof is completed.
\end{proof}

\begin{lem}
\label{lem2.6}
%Suppose that $\int_{1}^{\infty}y\upsilon(\mathrm{d}y)<\infty$.
It holds that $1\leq V_{0,b^{\omega}}^{\omega\prime}(x)\leq \phi$ for $x\in(0,\infty)$, and
\begin{eqnarray}\label{ver.xles.a}
\begin{cases}
\mathcal{A}V_{0,b^{\omega}}^{\omega}(x)-q V_{0,b^{\omega}}^{\omega}(x){\color{black}+\lambda\omega(x)}=0, \quad x\in(0, b^{\omega}],\\
\mathcal{A}V_{0,b^{\omega}}^{\omega}(x)-q V_{0,b^{\omega}}^{\omega}(x){\color{black}+\lambda\omega(x)}\leq 0,
\quad x\in(b^{\omega},\infty).
\end{cases}
\end{eqnarray}
\end{lem}

\begin{proof}
That $1\leq V_{0,b^{\omega}}^{\omega\prime}(x)\leq \phi$ for $x\in(0,\infty)$ is a direct consequence of Lemma \ref{lem2.2}.
Put $\kappa:=\tau_{0}^{-}\wedge\tau_{b^{\omega}}^{+}$. Under the dividend and capital injection strategy $(D^{(0,b^{\omega})}_t,R^{(0,b^{\omega})}_t)$, neither dividends will be paid out of the surplus process nor capitals will be injected into the surplus process prior to the time $\kappa$, hence the controlled process $U^{(0,b^{\omega})}_t$ follows the same dynamics of $X$ before $\kappa$. By the strong Markov property of the process {\color{black}$X$}, we have that
\begin{eqnarray}
&&
\mathrm{E}_{x}\left[\int_{0}^{\infty}\mathrm{e}^{-qt}\mathrm{d}D^{(0,b^{\omega})}_t
-\phi
\int_{0}^{\infty}\mathrm{e}^{-qt}\mathrm{d}R^{(0,b^{\omega})}_t+\lambda\int_{0}^{\infty}\mathrm{e}^{-qt}\omega(U_t^{(0,b^{\omega})})\mathrm{d}t
\bigg|\mathcal{F}_{s\wedge \kappa}
\right]\nonumber\\
\hspace{-0.3cm}
&=&\hspace{-0.3cm}
\mathrm{E}_{x}\left[\int_{0}^{\infty}\mathrm{e}^{-q(t+s\wedge \kappa)}
\mathrm{d}(D^{(0,b^{\omega})}_{t+s\wedge \kappa}-\phi R^{(0,b^{\omega})}_{t+s\wedge \kappa})+\lambda\int_{0}^{\infty}\mathrm{e}^{-q(t+s\wedge \kappa)}\omega(U_{t+s\wedge \kappa}^{(0,b^{\omega})})\mathrm{d}t\bigg|\mathcal{F}_{s\wedge \kappa}\right]
\nonumber\\
\hspace{-0.3cm}
&&\hspace{-0.3cm}
+\lambda\int_{0}^{s\wedge \kappa}\mathrm{e}^{-qt}\omega(X_t)\mathrm{d}t
\nonumber\\
\hspace{-0.3cm}
&=&\hspace{-0.3cm}\mathrm{e}^{-q(s\wedge \kappa)}\mathrm{E}_{X_{s\wedge \kappa}}\left[\int_{0}^{\infty}\mathrm{e}^{-qt}
\mathrm{d}(D^{(0,b^{\omega})}_{t}-\phi R^{(0,b^{\omega})}_{t})+\lambda\int_{0}^{\infty}\mathrm{e}^{-qt}\omega(U_{t}^{(0,b^{\omega})})\mathrm{d}t\right]
\nonumber\\
\hspace{-0.3cm}
&&\hspace{-0.3cm}
+\lambda\int_{0}^{s\wedge \kappa}\mathrm{e}^{-qt}\omega(X_t)\mathrm{d}t
\nonumber\\
\hspace{-0.3cm}
&=&\hspace{-0.3cm}\mathrm{e}^{-q(s\wedge \kappa)}V^{\omega}_{0,b^{\omega}}(X_{s\wedge \kappa}) +\lambda\int_{0}^{s\wedge \kappa}\mathrm{e}^{-qt}\omega(X_t)\mathrm{d}t,\quad  s\geq0,\,\,x\in(0,b^{\omega}),\nonumber
\end{eqnarray}
which implies that the right-hand side of the above equation is a martingale.
%By Lemma \ref{lem.2.4}, we have that the process $\{e^{-q(t\wedge\tau_{b^{\omega}}^{+}\wedge\tau_{0}^{-})}V_{0,b^{\omega}}^{\omega}(X_{t\wedge\tau_{b^{\omega}}^{+}\wedge\tau_{0}^{-}});t\geq0\}$ is a martingale.
%Indeed, for $x\in(d_{i},d_{i+1})$ and  $\sigma:=\sigma_{d_{i}}^{-}\wedge \sigma_{d_{i+1}}^{+}$, It\^{o}'s formula gives
By It\^{o}'s formula, it holds that
{\color{black}\begin{eqnarray}
&&e^{-q(t\wedge\kappa)}V_{0,b^{\omega}}^{\omega}(X_{t\wedge\kappa})
+\lambda\int_{0}^{t\wedge \kappa}\mathrm{e}^{-qs}\omega(X_s)\mathrm{d}s
-V_{0,b^{\omega}}^{\omega}(x)\nonumber\\
\hspace{-0.3cm}
&=&\hspace{-0.3cm}
\int_{0-}^{t\wedge\kappa}e^{-q s}\left((\mathcal{A}-q)V_{0,b^{\omega}}^{\omega}(X_{s-})+\lambda\omega(X_{s-})\right)\mathrm{d}s+\int_{0-}^{t\wedge\kappa}\sigma e^{-q s}V_{0,b^{\omega}}^{\omega\prime}(X_{s-})\mathrm{d}B_{s}
\nonumber\\
&&
+\int_{0-}^{t\wedge\kappa} e^{-q s}V_{0,b^{\omega}}^{\omega\prime}(X_{s-})
\mathrm{d}\big(X_{s}+\gamma s-\sigma B_{s}-\sum_{r\leq s}\Delta X_{r}\mathbf{1}_{\{\Delta X_{r}\geq 1\}}\big)
\nonumber\\
&&+\int_{0-}^{t\wedge\kappa}\int_{0}^{\infty}e^{-q s}\big[V_{0,b^{\omega}}^{\omega}(X_{s-}-y)-V_{0,b^{\omega}}^{\omega}(X_{s-})
%\nonumber\\
%&&
%\quad \quad \quad \quad \quad \quad \quad \quad
+V_{0,b^{\omega}}^{\omega\prime}(X_{s-})y\mathbf{1}_{(0,1]}(y)\big]
\overline{N}(\mathrm{d}s,\mathrm{d}y),\quad t\geq0.\nonumber
\end{eqnarray}
Following the same arguments in the proof of Lemma \ref{lem2.7}, we get that all the terms (except for the first one) on the right hand side of the above equality are martingales starting from $0$.} Hence, by taking expectations on both sides of the above equation, we get that
$$
0=\mathrm{E}_{x}\left[\int_{0-}^{t\wedge\kappa}e^{-q s}\left((\mathcal{A}-q)V_{0,b^{\omega}}^{\omega}(X_{s-})+\lambda\omega(X_{s-})\right)\mathrm{d}s\right],\quad t\geq0,\,\,x\in(0,b^{\omega}).
$$
Dividing both sides of the above equation by $t$ and then setting $t\downarrow 0$, we can obtain the equality in (\ref{ver.xles.a}) for $x\in(0,b^{\omega})$ by the mean value theorem and the dominated convergence theorem. In addition, the equality in (\ref{ver.xles.a}) for $x=b^{\omega}$ follows from the continuity of $(\mathcal{A}-q)V_{0,b^{\omega}}^{\omega}(x){\color{black}+\lambda\omega(x)}$ at $b^{\omega}$. For a more detailed proof of (\ref{ver.xles.a}), we refer to the proof of Lemma 4.2 in \cite{KypRS2010}.

It remains to prove the inequality in \eqref{ver.xles.a}. By the definition of the value function $V_{0,b^{\omega}}^{\omega}(x)$, we get that
\begin{eqnarray}
\label{2.30}
(\mathcal{A}-q)V_{0,b^{\omega}}^{\omega}(x){\color{black}+\lambda\omega(x)}=-\gamma+\int_{1}^{\infty}y\upsilon(\mathrm{d}y)-q(x-b^{\omega}+V_{0,b^{\omega}}^{\omega}(b^{\omega})){\color{black}+\lambda\omega(x)},\quad x\in(b^{\omega},\infty),
\nonumber
\end{eqnarray}
which, combined with the concavity of $\omega$ and the fact that $b^{\omega}>\sup\{x\geq0; q-\lambda\omega_{+}^{\prime}(x)\leq 0\}\vee Z_{q}^{-1}(\phi)>0$ (see; Lemma \ref{lem2.1}), yields that
\begin{eqnarray}
\label{slop.con.}
\left[(\mathcal{A}-q)V_{0,b^{\omega}}^{\omega}(x){\color{black}+\lambda\omega(x)}\right]^{\prime}
=-q+\lambda\omega_{+}^{\prime}(x)\leq 0, \quad x\in(b^{\omega},\infty).
\end{eqnarray}
In view of the fact that $(\mathcal{A}-q)V_{0,b^{\omega}}^{\omega}(b^{\omega})+\lambda\omega(b^{\omega})=0$ and \eqref{slop.con.}, we conclude the desired inequality of \eqref{ver.xles.a}. %The proof is complete.
\end{proof}

\begin{thm}
\label{thm.3.1}
%Suppose that $\int_{1}^{\infty}y\upsilon(\mathrm{d}y)<\infty$.
The double barrier dividend and capital injection strategy with the upper barrier $b^{\omega}$ and the lower barrier $0$ dominates all admissible singular dividend and capital injection strategies that
$$V_{0,b^{\omega}}^{\omega}(x)=\sup_{D,R}V_{D,R}^{\omega}(x).$$
\end{thm}

\begin{proof}
The desired conclusion is a direct consequence of Lemmas \ref{lem2.1}, \ref{lem2.2}, \ref{lem2.7} and \ref{lem2.6}.
\end{proof}

\section{Optimality of Regime-modulated Double Barrier Strategy}\label{sec:opt}

We continue to prove the main result Theorem \ref{thm4.1} using results from the previous auxiliary control problem with a final payoff and the recursive iteration based on dynamic programming principle. As preparations, let us first consider the following space of functions
\begin{eqnarray}
\mathcal{B}\hspace{-0.3cm}&:=
&\hspace{-0.3cm}\{\left.f: \mathbb{R}_{+}\times \mathcal{E}\rightarrow \mathbb{R}\,\right|\, \text{for each }  %x\geq 0 \,\text{ and } \,
i\in\mathcal{E}, %\, f(x,i)\in C([0,\infty)),\,
\text{ the function } x\mapsto f(x,i)-x  \nonumber\\
\hspace{-0.3cm}&
&\hspace{-0.3cm}\quad
\text{ is continuous and bounded over } [0,\infty)\},\nonumber
\end{eqnarray}
endowed with the metric
$$\rho(f,g):=\max_{i\in \mathcal{E}}\sup_{x\geq 0}\left|f(x,i)-g(x,i)\right|
=\max_{i\in \mathcal{E}}\sup_{x\geq 0}\left|\left(f(x,i)-x\right)-\left(g(x,i)-x\right)\right|.$$
It is straightforward to check that the metric space $(\mathcal{B},\rho)$ is complete.

The following Lemma \ref{lem.4.0} states that the value function $V$ is an element of $\mathcal{B}$.

\begin{lem}
\label{lem.4.0}
Denote %$\overline{\delta}:=\max_{i\in\mathcal{E}}\delta(i)$,
$\underline{\delta}:=\min_{i\in\mathcal{E}}\delta_{i}$, $\overline{X}_{t}:=\sup_{s\leq t}X_{s}$, and $\underline{X}_{t}:=\inf_{s\leq t}X_{s}$.
We have that
$$\underline{V}(x,i):=x+\phi\mathrm{E}_{0,i}\left[\int_{0}^{\infty}e^{-\underline{\delta}t}\mathrm{d}\left(\underline{X}_t\wedge 0\right)\right]\leq V(x,i)\leq x+\mathrm{E}_{0,i}\left[\int_{0}^{\infty}e^{-\underline{\delta}t}\mathrm{d}\left(\overline{X}_t\vee 0\right)\right]=:\overline{V}(x,i),$$
for all $(x,i)\in\mathbb{R}_{+}\times\mathcal{E}$.
\end{lem}
\begin{proof}
The lower bound of $V(x.i)$ can be derived if we consider the extreme admissible dividend and capital injection strategy where the manager of the company pays whatever she has as dividends at time $0$, and pays no dividends afterwards and bail out all deficits by injecting capitals.
The upper bound is derived by considering an admissible dividend and capital injection strategy $(\Tilde{D},\Tilde{R})$ where the manager of the company pays every dollar accumulated by $X$ as dividends as early as possible all the way (i.e., $\Tilde{D}_{t}:=\overline{X}_t\vee 0$), and cover all deficits by capital injection (i.e., $\Tilde{R}_{t}:=-\inf_{s\leq t}(X_{s}-(\overline{X}_t\vee 0))$). 
Furthermore, for any an admissible dividend and capital injection strategy $({D},{R})$, by integrating by parts, we have
\begin{eqnarray}\label{4.1}
\hspace{-0.3cm}&&\hspace{-0.3cm}
\mathrm{E}_{x,i}\left[\int_0^{\infty}e^{-\int_0^t\delta_{Y_s}\mathrm{d}s}\mathrm{d}\Tilde{D}_t\right]-\mathrm{E}_{x,i}\left[\int_0^{\infty}e^{-\int_0^t\delta_{Y_s}\mathrm{d}s}\mathrm{d}{D}_t-\phi\int_0^{\infty}e^{-\int_0^t\delta_{Y_s}\mathrm{d}s}\mathrm{d}{R}_t\right]
%-V_{{D},{R}}(x,i)
\nonumber\\
\hspace{-0.3cm}&=&\hspace{-0.3cm}
%\mathrm{E}_{x,i}\left[\int_0^{\infty}e^{-\int_0^t\delta_{Y_s}\mathrm{d}s}\mathrm{d}\left(\Tilde{D}_t-{D}_t+\phi{R}_t\right)\right]=
\mathrm{E}_{x,i}\left[\int_{0+}^{\infty}\delta_{Y_t}e^{-\int_0^t\delta_{Y_s}\mathrm{d}s}\left(\Tilde{D}_t-{D}_t+\phi{R}_t\right)\mathrm{d}t\right]
\nonumber\\
\hspace{-0.3cm}&\geq&\hspace{-0.3cm}
\mathrm{E}_{x,i}\left[\int_{0+}^{\infty}\delta_{Y_t}e^{-\int_0^t\delta_{Y_s}\mathrm{d}s}\left(X_t-{D}_t+{R}_t\right)\mathrm{d}t\right]\geq0,
\end{eqnarray}
where the first inequality follows from the facts that $\Tilde{D}_t=\overline{X}_t\vee0\geq X_t$, $R_t\geq0$ and $\phi>1$, and the second inequality is because the strategy $(D,R)$ is admissible.
Then, using \eqref{4.1} and the arbitrariness of the admissible strategy $(D,R)$ implies that
\begin{eqnarray}
V(x,i)
\leq
\mathrm{E}_{x,i}\left[\int_0^{\infty}e^{-\int_0^t\delta_{Y_s}\mathrm{d}s}\mathrm{d}\Tilde{D}_t\right]
\hspace{-0.3cm}&=&\hspace{-0.3cm}
\mathrm{E}_{x,i}\left[\int_0^{\infty}e^{-\int_0^t\delta_{Y_s}\mathrm{d}s}\mathrm{d}\left(\overline{X}_t\vee 0\right)\right]
\nonumber\\
\hspace{-0.3cm}&=&\hspace{-0.3cm}
x+\mathrm{E}_{0,i}\left[\int_0^{\infty}e^{-\int_0^t\delta_{Y_s}\mathrm{d}s}\mathrm{d}\left(\overline{X}_t\vee 0\right)\right]\leq \overline{V}(x,i),
\end{eqnarray}
where the second equality has used the spatial homogeneity of L\'evy processes.
%The upper bound is derived by considering the extreme case where the manager of the company pays every dollar accumulated by $X$ as dividends as early as possible all the way (i.e., $D_{t}:=\overline{X}_t\vee 0$), and cover all deficits by capital injection (i.e., $R_{t}:=-\inf_{s\leq t}(X_{s}-(\overline{X}_t\vee 0))$). We note that the resulting surplus process
%$U_t=X_{t}-(\overline{X}_t\vee 0)-\inf_{s\leq t}(X_{s}-(\overline{X}_{s}\vee 0))$ can take positive values, however, such positive surplus is accumulated completely due to injected capitals and hence it is not reasonable to pay it as dividends because $\phi>1$. Hence, $D_{t}=\overline{X}_t\vee 0$ amounts to the maximum reasonable amount of dividends paid until time $t\geq0$. 
This completes the proof.
\end{proof}

For any function $f: [0,\infty)\times \mathcal{E}\rightarrow \mathbb{R}$, we define a function $\widehat{f}: [0,\infty)\times \mathcal{E}\rightarrow \mathbb{R}$ that
\begin{eqnarray}
\widehat{f}(x,i):=\sum_{j\in\mathcal{E}, j\neq  i}\frac{\lambda_{ij}}{\lambda_{i}}\int_{-\infty}^{0}\left[f(x+y,j)\mathbf{1}_{\{-y\leq x\}}+(\phi(x+y)+f(0,j))\mathbf{1}_{\{-y>x\}}\right]\mathrm{d}F_{ij}(y),
\end{eqnarray}
where $\lambda_{i}=\sum_{j\neq i}\lambda_{ij}$, and $F_{ij}$ is the distribution function of $J_{ij}$ for $i,j\in\mathcal{E}$.
Note that
\begin{eqnarray}
\label{4.2}
\left|\widehat{f}(x,i)-x\right|
\hspace{-0.3cm}&=&\hspace{-0.3cm}
\left|\sum_{j\in\mathcal{E}, j\neq  i}\frac{\lambda_{ij}}{\lambda_{i}}\int_{-\infty}^{0}\left[\left(f(x+y,j)-x-y\right)\mathbf{1}_{\{-y\leq x\}}+\left(f(0,j)-0\right)\mathbf{1}_{\{-y>x\}}\right.\right.
\nonumber\\
\hspace{-0.3cm}&&\hspace{-0.3cm}
\left.\left.+\phi\left(x+y\right)\mathbf{1}_{\{-y>x\}}+y\mathbf{1}_{\{-y\leq x\}}-x\mathbf{1}_{\{-y>x\}}\right]\mathrm{d}F_{ij}(y)\right|
\nonumber\\
\hspace{-0.3cm}&\leq&\hspace{-0.3cm}
\sum_{j\in\mathcal{E}, j\neq  i}\frac{\lambda_{ij}}{\lambda_{i}}\Big[\rho(f,x)+(\phi+1)\mathrm{E}\left|J_{ij}\right|
%+\left|\int_{-x}^{0}x\mathrm{d}F_{ij}(y)-x\right|
\Big]
,\quad (x,i)\in [0,\infty)\times\mathcal{E},
\end{eqnarray}
where we have used the fact that $0<\left|x\right|\vee \left|x+y\right|<\left|y\right|$ on $\{-y>x\}$.
%, and\begin{eqnarray}\left|\int_{-x}^{0}x\mathrm{d}F_{ij}(y)-x\right|\hspace{-0.3cm}&=&\hspace{-0.3cm}\left|x\int_{-\infty}^{-x}\mathrm{d}F_{ij}(y)\right|<\left|\int_{-\infty}^{-x}y\mathrm{d}F_{ij}(y)\right|\leq -\mathbf{E}[J_{ij}].\end{eqnarray}
By \eqref{4.2} and $\max_{i,j\in\mathcal{E}}\mathrm{E}\left|J_{ij}\right|<\infty$, one gets that $\widehat{f} \in \mathcal{B}$ when $f\in \mathcal{B}$.

For any function $\mathbf{b}=(b_{i})\in[0,\infty)^{\mathcal{E}}$, denote by $V_{0,\mathbf{b}}(x,i)$ the value function (i.e., the NPV of the accumulated differences between dividends and the costs of capital injections) of the double barrier dividend and capital injection strategy with dynamic upper barrier $b_{Y_t}$ and constant lower barrier $0$. In addition, let us define a mapping $T_{\mathbf{b}}$ acting on $f \in \mathcal{B}$ such that
\begin{eqnarray}
\label{def.Tb}
T_{\mathbf{b}}f(x,i)
\hspace{-0.3cm}&:=&\hspace{-0.3cm}
\mathrm{E}_x^i\bigg[\int_{0}^{\infty}e^{-q_{i}t}\mathrm{d}D_{t}^{b_{i},i}
-\phi\int_{0}^{\infty}e^{-q_{i}t}\mathrm{d}R_{t}^{b_{i},i}+\lambda_{i}\int_{0}^{\infty}e^{-q_{i}t}\widehat{f}(U_{t}^{b_{i},i},i)\mathrm{d}t\bigg],
\end{eqnarray}
where $q_i=\delta_{i}+\lambda_i$ and $\mathrm{E}_x^i$ denotes the expectation operator with respect to the law of the process $X^{i}$ conditioned on the event $\{X_{0}^{i}=x\}$. The process $U_{t}^{b_{i},i}$ is the double-reflected process with upper reflecting barrier $b_{i}\geq0$, lower reflecting barrier $0$, and the underlying risk process $X^{i}$; and $D_{t}^{b_{i},i}$, $R_{t}^{b_{i},i}$ are the cumulative dividends paid and capitals injected, respectively.
In what follows, the scale functions of $X^{i}$ will be denoted by $W_{q,i}$, $Z_{q,i}$ and $\overline{Z}_{q,i}$, whose definitions are given in Section 2.2 where the subscript $i$ is absent.
%In addition,   under the probability $\mathrm{P}_{x,i}$.

\begin{lem}\label{lem4.1}
For $\mathbf{b}\in[0,\infty)^{\mathcal{E}}$ and $(x,i)\in R \times \mathcal{E}$, we have
$V_{0,\mathbf{b}}(x,i)=T_{\mathbf{b}}V_{0,\mathbf{b}}(x,i)$.
\end{lem}
\begin{proof}
When $Y_{0}=i$, let $e_{\lambda_{i}}$ be the first time $Y$ switches the regime state.
By the Markov property, the proof of Proposition \ref{prop2.1}, Theorem 1 in \cite{Pistorius03}, as well as the independence between $(X^{i})_{i\in\mathcal{E}}$, $Y$ and $(J_{ij})_{i,j\in\mathcal{E}}$, we can derive
\begin{eqnarray}
V_{0,\mathbf{b}}(x,i)
\hspace{-0.3cm}&=&\hspace{-0.3cm}
\mathrm{E}_{x,i}\bigg[\int_{0}^{e_{\lambda_{i}}}e^{-\delta_{i}t}\mathrm{d}D_{t}^{b_{i},i}
-\phi\int_{0}^{e_{\lambda_{i}}}e^{-\delta_{i}t}\mathrm{d}R_{t}^{b_{i},i}+e^{-\delta_{i}e_{\lambda_{i}}}V_{0,\mathbf{b}}(U_{e_{\lambda_{i}}}^{b_{i},i}+J_{iY_{e_{\lambda_{i}}}},Y_{e_{\lambda_{i}}})\bigg]
\nonumber\\
\hspace{-0.3cm}&=&\hspace{-0.3cm}
\mathrm{E}_x^i\bigg[\int_{0}^{\infty}e^{-q_{i}t}\mathrm{d}D_{t}^{b_{i},i}
-\phi\int_{0}^{\infty}e^{-q_{i}t}\mathrm{d}R_{t}^{b_{i},i}\bigg]
+\sum_{j\neq i}\lambda_{ij}\mathrm{E}_x^i\bigg[\int_{0}^{\infty} e^{-q_{i} t}V_{0,\mathbf{b}}(U_{t}^{b_{i},i}+J_{ij},j)
\mathrm{d}t\bigg]
\nonumber\\
\hspace{-0.3cm}&=&\hspace{-0.3cm}
-\overline{Z}_{q_{i},i}(b_{i}-x)-\frac{\psi_{i}^{\prime}(0+)}{q_{i}}+\frac{Z_{q_{i},i}(b_{i})}{q_{i}W_{q_{i},i}(b_{i})}Z_{q_{i},i}(b_{i}-x)-\frac{\phi Z_{q_{i},i}(b_{i}-x)}{q_{i}W_{q_{i},i}(b_{i})}
\nonumber\\
\hspace{-0.3cm}&&\hspace{-0.3cm}
+\sum_{j\neq i}\lambda_{ij}\int_{-\infty}^{0}\Bigg[
\int_{0}^{b_{i}}V_{0,\mathbf{b}}(y+z,j)\Bigg[\frac{Z_{q_{i},i}(b_{i}-x)W_{q_{i},i}^{\prime+}(y)}{q_{i}W_{q_{i},i}(b_{i})}-W_{q_{i},i}(y-x)\Bigg]\mathrm{d}y
\nonumber\\
\hspace{-0.3cm}&&\hspace{-0.3cm}
+V_{0,\mathbf{b}}(z,j)\frac{Z_{q_{i},i}(b_{i}-x)W_{q_{i},i}(0+)}{q_{i}W_{q_{i},i}(b_{i})}
\Bigg]\mathrm{d}F_{ij}(z),\quad x\in[0,b_{i}], \,\,q_{i}=\delta_{i}+\lambda_{i},
\label{eq.4.4}
\\
V_{0,\mathbf{b}}(x,i)\hspace{-0.3cm}&=&\hspace{-0.3cm}
\left(x-b_{i}+V_{0,\mathbf{b}}(b_{i},i)\right)\mathbf{1}_{[b_{i},\infty)}(x)+
\left(\phi x+V_{0,\mathbf{b}}(0,i)\right)\mathbf{1}_{(-\infty,0)}(x).
\label{eq.4.5}
\end{eqnarray}
Using the above result, the continuity of scale functions, the boundedness of $V_{0,\mathbf{b}}$ over $[0,b_{i}]$ (see; Lemma \ref{lem.4.0}) as well as $\max_{j\neq i}\mathrm{E}|J_{ij}|<\infty$, we can deduce that $V_{0,\mathbf{b}}\in \mathcal{B}$.
By the definition of $T_{\mathbf{b}}$ in \eqref{def.Tb}, the second equality in \eqref{eq.4.4}, the independence between $U_{t}^{b_{i},i}$ and $J_{ij}$ for all $i,j\in\mathcal{E}$ as well as the fact that
\begin{eqnarray}
V_{0,\mathbf{b}}(U_{t}^{b_{i},i}+J_{ij},j)=V_{0,\mathbf{b}}(U_{t}^{b_{i},i}+J_{ij},j)\mathbf{1}_{\{U_{t}^{b_{i},i}\geq-J_{ij}\}}
+
\left(V_{0,\mathbf{b}}(0,j)+\phi\left(U_{t}^{b_{i},i}+J_{ij}\right)\right)\mathbf{1}_{\{U_{t}^{b_{i},i}<-J_{ij}\}},
\nonumber
\end{eqnarray}
we can finally conclude that $V_{0,\mathbf{b}}(x,i)=T_{\mathbf{b}}V_{0,\mathbf{b}}(x,i)$. %The proof is then completed.
\end{proof}

\begin{lem}\label{lem4.2}
The operator $T_{\mathbf{b}}$ is a contraction on $\mathcal{B}$ under the metric $\rho(\cdot,\cdot)$. In particular, for $f \in \mathcal{B}$, we have that
\begin{eqnarray}\label{V0b.lim}
V_{0,\mathbf{b}}(x,i)=\lim\limits_{n\uparrow\infty}T_{\mathbf{b}}^{n}f(x,i),\quad (x,i)\in [0,\infty)\times\mathcal{E},
\end{eqnarray}
where the convergence is under the metric $\rho(\cdot,\cdot)$ and $T_{\mathbf{b}}^{n}(f):=T_{\mathbf{b}}(T_{\mathbf{b}}^{n-1}(f))$ for $n>1$ with $T_{\mathbf{b}}^{1}:=T_{\mathbf{b}}$.
\end{lem}
\begin{proof}
Recall that the metric space $(\mathcal{B},\rho)$ is complete. By Proposition \ref{prop2.1}, for $f\in \mathcal{B}$, we deduce that
\begin{eqnarray}
\label{ref4.7}
\big|T_{\mathbf{b}}f(x,i)-x\big|
\hspace{-0.3cm}&=&\hspace{-0.3cm}
\bigg|\mathrm{E}_x^i\bigg[\int_{0}^{\infty}e^{-q_{i}t}\mathrm{d}D_{t}^{b_{i},i}
-\phi\int_{0}^{\infty}e^{-q_{i}t}\mathrm{d}R_{t}^{b_{i},i}+\lambda_{i}\int_{0}^{\infty}e^{-q_{i}t}\widehat{f}(U_{t}^{b_{i},i},i)\mathrm{d}t\bigg]-x\bigg|
\nonumber\\
\hspace{-0.3cm}&=&\hspace{-0.3cm}
\bigg|-\overline{Z}_{q_{i},i}(b_{i}-x)-\frac{\psi_{i}^{\prime}(0+)}{q_i}+\frac{\lambda_i}{q_i}\bigg[\widehat{f}(0,i)+\int_{0}^{b_i}\widehat{f}^{\prime}(y,i)Z_{q_{i},i}(y-x)\mathrm{d}y\bigg]
\nonumber\\
\hspace{-0.3cm}&&\hspace{-0.3cm}
+\frac{Z_{q_{i},i}(b_{i}-x)}{q_{i}W_{q_{i},i}(b_i)}\Big[Z_{q_{i},i}(b_i)-\phi-\lambda_{i}\int_{0}^{b_i}\widehat{f}^{\prime}(y,i)W_{q_{i},i}(y)\mathrm{d}y\Big]-x\bigg|,\quad x\in [0,b_i],
\end{eqnarray}
and
\begin{eqnarray}
\label{ref4.8}
|T_{\mathbf{b}}f(x,i)-x|\hspace{-0.3cm}&=&\hspace{-0.3cm}\left|-b_i
-\frac{\psi_{i}^{\prime}(0+)}{q_i}+\frac{\lambda_i}{q_i}\bigg[\widehat{f}(0,i)+\int_{0}^{b_i}\widehat{f}^{\prime}(y,i)\mathrm{d}y\bigg]\right.
\nonumber\\
\hspace{-0.3cm}&&\hspace{-0.3cm}
+\left.\frac{1}{q_{i}W_{q_{i},i}(b_i)}\bigg[Z_{q_{i},i}(b_i)-\phi-\lambda_{i}\int_{0}^{b_i}\widehat{f}^{\prime}(y,i)W_{q_{i},i}(y)\mathrm{d}y\bigg]
\right|,\quad x \in (b_i,\infty).
\end{eqnarray}
By \eqref{ref4.7}, \eqref{ref4.8}, and the fact that $\widehat{f}\in \mathcal{B}$, it holds that  $T_{\mathbf{b}}f\in \mathcal{B}$ for $f\in \mathcal{B}$. Moreover, for $f,g \in \mathcal{B}$, we can get that
\begin{eqnarray}\label{metric}
\rho(T_{\mathbf{b}}f,T_{\mathbf{b}}g)
\hspace{-0.3cm}&=&\hspace{-0.3cm}
\sup_{i\in \mathcal{E},\,x\geq 0}\mathrm{E}_x^i\Bigg[
e^{-\delta_{i}e_{\lambda_i}}\sum_{j\in\mathcal{E}, j\neq  i}\frac{\lambda_{ij}}{\lambda_{i}}\int_{-\infty}^{-U^{b_{i},i}_{e_{\lambda_i}}}\left|f(0,j)-g(0,j)\right|\mathrm{d}F_{ij}(y)
\nonumber\\
\hspace{-0.3cm}&&\hspace{-0.3cm}
+e^{-\delta_{i}e_{\lambda_i}}\sum_{j\in\mathcal{E}, j\neq  i}\frac{\lambda_{ij}}{\lambda_{i}}\int_{-U^{b_{i},i}_{e_{\lambda_i}}}^{0}\left|f(U^{b_{i},i}_{e_{\lambda_i}}+y,j)-g(U^{b_{i},i}_{e_{\lambda_i}}+y,j)\right|\mathrm{d}F_{ij}(y)
\Bigg]
%\nonumber\\\hspace{-0.3cm}&\leq&\hspace{-0.3cm}\rho(f,g)\sup_{i\in \mathcal{E}}\mathrm{E}_{0,i}\bigg[e^{-\delta_{i}e_{\lambda_i}}\sum_{j\in\mathcal{E}, j\neq  i}\frac{\lambda_{ij}}{\lambda_{i}}\int_{(-\infty,0]}\mathrm{d}F_{ij}(y)\bigg]
\nonumber\\
\hspace{-0.3cm}&\leq&\hspace{-0.3cm}
\rho(f,g)\sup_{i\in \mathcal{E}}\mathrm{E}_0^i\big[e^{-\delta_{i}e_{\lambda_i}}\big]
\nonumber\\
\hspace{-0.3cm}&:=&\hspace{-0.3cm}
\beta \rho(f,g),\quad \beta\in(0,1).
\end{eqnarray}
By (\ref{metric}), for $f\in \mathcal{B}$, $(T_{\mathbf{b}}^nf)_{n\geq1}$ is a Cauchy sequence. Hence, we have that
$$T_{\mathbf{b}}^{\infty}f:=\lim\limits_{n\uparrow \infty}T_{\mathbf{b}}^nf=T_{\mathbf{b}}(\lim\limits_{n\uparrow\infty}T_{\mathbf{b}}^nf)=T_{\mathbf{b}}(T_{\mathbf{b}}^{\infty}f), \quad f\in \mathcal{B},$$
which implies that $T_{\mathbf{b}}^{\infty}f$ is a fixed point of the mapping $T_{\mathbf{b}}$. By lemma \ref{lem4.1}, we obtain (\ref{V0b.lim}) as desired.
\end{proof}

Let us define another space of functions that
\begin{eqnarray}
\mathcal{C}:=\{\left.f\in\mathcal{B}\,\right|\,\widehat{f}(x,i) \text{ is concave and }1\leq \widehat{f}^{\prime}(x,i)\leq \phi \text{ for all }x\in[0,\infty)\text{ and } i\in\mathcal{E}\}.
\end{eqnarray}

\begin{lem}\label{lem4.3}
Suppose that $f\in\mathcal{B}\cap C^{1}(\mathbb{R}_{+})$ is concave, non-decreasing, and satisfies $1\leq f^{\prime}(x,i)\leq \phi$ over $\mathbb{R}_{+}\times\mathcal{E}$, we have that $f\in \mathcal{C}$.
\end{lem}

\begin{proof}
By definition, $\widehat{f}$ can be rewritten as
\begin{eqnarray}
\widehat{f}(x,i)
\hspace{-0.3cm}&=&\hspace{-0.3cm}
\sum_{j\in\mathcal{E}, j\neq  i}\frac{\lambda_{ij}}{\lambda_{i}}\bigg[\int_{-x}^{0}\big[f(x+y,j)-(\phi(x+y)+f(0,j))\big]\mathrm{d}F_{ij}(y)+\phi(x+\mathrm{E}[J_{ij}])+f(0,j)\bigg],
%\nonumber\\
%\hspace{-0.3cm}&=&\hspace{-0.3cm}
%\sum_{j\in\mathcal{E}, j\neq  i}\frac{\lambda_{ij}}{\lambda_{i}}\bigg[f(x,j)+\phi\mathrm{E}[J_{ij}]-\int_{-x}^{0}\big[f^{\prime}(x+y,j)-\phi\big]F_{ij}(y)\mathrm{d}y\bigg],
\nonumber
\end{eqnarray}
which implies that
\begin{eqnarray}
\label{f3.3}
\widehat{f}^{\prime}(x,i)=\sum_{j\in\mathcal{E}, j\neq  i}\frac{\lambda_{ij}}{\lambda_{i}}\bigg[\phi+\int_{-x}^{0}\big[f^{\prime}(x+y,j)-\phi\big]\mathrm{d}F_{ij}(y)\bigg].
\end{eqnarray}
This result, together with the concavity of $f$, yields the concavity of $\widehat{f}(x,i)$. In addition,
by \eqref{f3.3} and the fact that $1-\phi\leq f^{\prime}(x+y,j)-\phi\leq 0$, one can deduce that
\begin{eqnarray}
1=\sum_{j\in\mathcal{E}, j\neq  i}\frac{\lambda_{ij}}{\lambda_{i}}
\leq
\sum_{j\in\mathcal{E}, j\neq  i}\frac{\lambda_{ij}}{\lambda_{i}}\bigg[\phi+\int_{-x}^{0}(1-\phi)\mathrm{d}F_{ij}(y)\bigg]\leq
\widehat{f}^{\prime}(x,i)\leq \phi \sum_{j\in\mathcal{E}, j\neq  i}\frac{\lambda_{ij}}{\lambda_{i}}=\phi.\nonumber
\end{eqnarray}
The proof is completed.
\end{proof}

For $f\in \mathcal{C}$ and $(x,i)\in \mathbb{R}_{+} \times \mathcal{E}$, let us define another operator $T_{\sup}$ that
\begin{eqnarray}\label{aux.oper}
T_{\sup} f(x,i)
\hspace{-0.3cm}&:=&\hspace{-0.3cm}
\sup_{D,R}\mathrm{E}_{x,i}\bigg[\int_{0}^{e_{\lambda_{i}}}e^{-\delta_{i}t}\mathrm{d}D_t
-\phi\int_{0}^{e_{\lambda_{i}}}e^{-\delta_{i}t}\mathrm{d}R_t
+e^{-\delta_{i}e_{\lambda_{i}}}\widehat{f}(U_{e_{\lambda_{i}}-},i)\bigg]
\nonumber\\
\hspace{-0.3cm}&=&\hspace{-0.3cm}
\sup_{D^i,R^i}\mathrm{E}_x^i\bigg[\int_{0}^{\infty}e^{-q_{i}t}\mathrm{d}D^{i}_{t}-\phi\int_{0}^{\infty}e^{-q_{i}t}\mathrm{d}R^{i}_t+\lambda_{i}\int_{0}^{\infty}e^{-q_{i}t}\widehat{f}(U^{i}_t,i)\mathrm{d}t\bigg],
\end{eqnarray}
where $U^{i}_{t}=X^{i}_{t}-D^{i}_{t}+R^{i}_{t}$ represents the controlled surplus process with control $(D^{i},R^{i})$ and driving process $X^{i}$.

We denote $\underline{V}_0:=\underline{V}$ and $\overline{V}_0:=\overline{V}$ as well as $\underline{V}_n:=T_{\sup} (\underline{V}_{n-1})$ and $\overline{V}_n:=T_{\sup} (\overline{V}_{n-1})$, for $n \geq 1$.

\begin{lem}\label{pro4.1}
%If $v_0^-\leq V_{0,b^{\ast}} \leq v_0^+$, then w
We have $\underline{V}_n\leq V\leq \overline{V}_n$ on $\mathbb{R}_{+}\times\mathcal{E}$ for all $n \geq 1$, and
\begin{eqnarray}\label{V.lim}
V(x,i)=\lim\limits_{n\uparrow\infty}\underline{V}_n(x,i)=\lim\limits_{n\uparrow\infty}\overline{V}_n(x,i),\quad (x,i) \in \mathbb{R}_{+}\times \mathcal{E},
\end{eqnarray}
where the convergence is under the metric $\rho(\cdot,\cdot)$. Moreover, we have $V\in \mathcal{C}$.
\end{lem}

\begin{proof}
The first claim of Lemma \ref{pro4.1} can be verified by the method of induction. In fact, by Lemma \ref{lem.4.0}, we have $\underline{V}_0\leq V\leq \overline{V}_0$. Suppose that $\underline{V}_{n-1}\leq V\leq \overline{V}_{n-1}$, then
$$\underline{V}_n=T_{\sup} \underline{V}_{n-1} \leq T_{\sup} V \leq T_{\sup} \overline{V}_{n-1}=\overline{V}_{n},$$
which, together with the fact that $V$ is a fixed point of the mapping $T_{\sup}$, implies the desired claim that $\underline{V}_n\leq V\leq \overline{V}_n$ for all $n\geq 1$.

To prove the second claim of Lemma \ref{pro4.1}, for any $f \in \mathcal{C}$ and $i \in \mathcal{E}$, Theorem \ref{thm.3.1} guarantees the existence of $b^{f}_i\in (0,\infty)$ such that the second equality of (\ref{aux.oper}) is achieved by the expected NPV under a double barrier strategy with upper barrier $b^{f}_{i}$ and lower barrier $0$. Denote $\mathbf{b}^{f}=(b^{f}_{i})_{i\in \mathcal{E}}$, it follows that $T_{\sup} f=T_{\mathbf{b}^{f}}f$ over $\mathbb{R}_{+}\times \mathcal{E}$, which, together with Lemmas \ref{lem2.1}-\ref{lem2.2}, implies that $T_{\sup} f(\cdot,i)\in C^{1}(0,\infty)$ and it is concave as well as $1\leq (T_{\sup} f)^{\prime}(\cdot,i)\leq \phi$ over $[0,\infty)$ for all $i\in\mathcal{E}$. Hence, $T_{\sup} f \in \mathcal{C}$ (see; Lemma \ref{lem4.3}), and
$$\rho(T_{\sup} f,T_{\sup} g)=\rho(T_{\mathbf{b}^{f}} f, T_{\mathbf{b}^{g}} g)
=\rho(\sup_{\mathbf{b}}T_{\mathbf{b}} f, \sup_{\mathbf{b}}T_{\mathbf{b}} g)
\leq \sup_{\mathbf{b}}\rho(T_{\mathbf{b}}f,T_{\mathbf{b}}g)\leq \beta \rho(f,g), \quad \beta\in(0,1),$$ i.e., $T_{\sup}$ is a contraction mapping from $\mathcal{C}$ to itself.
Hence, the Cauchy sequences $(\underline{V}_n)_{n\geq 1}$ and $(\overline{V}_n)_{n\geq 1}$ converge to the unique fixed point $V$ of $T_{\sup}$.
%By Lemma \ref{lem.4.0}, we have that $v_0^+,v_0^- \in \mathcal{C}$ is existent.
In addition, by (\ref{V.lim}) and the dominated convergence theorem, we have that
\begin{eqnarray}\label{V.hat}
\widehat{V}(x,i)
\hspace{-0.3cm}&=&\hspace{-0.3cm}
\sum_{j\in\mathcal{E}, j\neq  i}\frac{\lambda_{ij}}{\lambda_{i}}\int_{-\infty}^{0}\Big[V(x+y,j)\mathbf{1}_{\{-y\leq x\}}
+(\phi(x+y)+V(0,j))\mathbf{1}_{\{-y>x\}}\Big]\mathrm{d}F_{ij}(y)
\nonumber\\
\hspace{-0.3cm}&=&\hspace{-0.3cm}
\lim\limits_{n\rightarrow\infty}\sum_{j\in\mathcal{E}, j\neq  i}\frac{\lambda_{ij}}{\lambda_{i}}\int_{-\infty}^{0}\Big[\underline{V}_n(x+y,j)\mathbf{1}_{\{-y\leq x\}}
+(\phi(x+y)+\underline{V}_n(0,j))\mathbf{1}_{\{-y>x\}}\Big]\mathrm{d}F_{ij}(y)
\nonumber\\
\hspace{-0.3cm}&=&\hspace{-0.3cm}
\lim\limits_{n\rightarrow\infty}\sum_{j\in\mathcal{E}, j\neq  i}\frac{\lambda_{ij}}{\lambda_{i}}\int_{-\infty}^{0}\Big[\overline{V}_n(x+y,j)\mathbf{1}_{\{-y\leq x\}}
+(\phi(x+y)+\overline{V}_n(0,j))\mathbf{1}_{\{-y>x\}}\Big]\mathrm{d}F_{ij}(y)
\nonumber\\
\hspace{-0.3cm}&=&\hspace{-0.3cm}
\lim\limits_{n\rightarrow\infty}\widehat{\underline{V}_{n}}(x,i)
=\lim\limits_{n\rightarrow\infty}\widehat{\overline{V}_{n}}(x,i)
,\quad (x,i)\in \mathbb{R}_{+}\times \mathcal{E}.
\end{eqnarray}
By (\ref{V.hat}) and the facts that $(\underline{V}_n)_{n\geq 1}\subseteq \mathcal{C}$ and $(\overline{V}_n)_{n\geq 1}\subseteq \mathcal{C}$, we obtain that $V\in \mathcal{C}$.
\end{proof}

%According to Proposition \ref{pro4.1}, we can obtain an iterative construction of the value function as follows: Initializing by $n=0$ and $v=v_0$ for some $v_0 \in \mathcal{C}$, we can operate the iteration scheme by\begin{enumerate}[itemindent=1em]    \item [(a)]    finding ${\mathbf{b}}^v=(b^v(i);i\in \mathcal{E})$ as defined in the proof of Proposition \ref{pro4.1};    \item [(b)]   setting $T_{{\mathbf{b}}^v}\rightarrow v$, $n+1\rightarrow n$ and returning to step (a).\end{enumerate}
Finally, we can give the proof of Theorem \ref{thm4.1} using the previous preparations.

\begin{proof}[Proof of Theorem \ref{thm4.1}]
From Lemma \ref{pro4.1}, it follows that $V\in \mathcal{C}$. This result and Theorem \ref{thm.3.1} imply that there exists a function $\mathbf{b}^{V}=(b_{i}^{V})_{i\in\mathcal{E}}\in(0,\infty)^{\mathcal{E}}$ such that
$V(x,i)=T_{\sup}V(x,i)=T_{\mathbf{b}^{V}}V(x,i)$ for all $(x,i)\in\mathbb{R}_{+}\times \mathcal{E}$. Therefore, by \eqref{V0b.lim} and $V\in\mathcal{B}$ (as $V\in \mathcal{C}$), we have that
$$V(x,i)=\lim_{n\uparrow\infty}T_{\mathbf{b}^{V}}^{n}V(x,i)=V_{0,\mathbf{b}^{V}}(x,i),$$
i.e., $\mathbf{b}^{*}:=\mathbf{b}^{V}=(b_{i}^{V})_{i\in\mathcal{E}}$ is the desired barrier function such that the conclusion of Theorem \ref{thm4.1} holds.
\end{proof}

\ \\
{\small \textbf{Acknowledgements}:  W. Wang is supported by the National Natural Science Foundation of China under no. 12171405 and no. 11661074.
X. Yu is supported by the Hong Kong Polytechnic University research grant under no. P0031417.}

\end{document}